\documentclass[oneside,french,english,american]{amsart}
\usepackage[T1]{fontenc}
\usepackage[latin9]{inputenc}
\usepackage{geometry}
\usepackage{float}
\usepackage{amstext}
\usepackage{amsthm}
\usepackage{amssymb}
\usepackage{graphicx}
\usepackage[all]{xy}

\makeatletter

\newcommand{\noun}[1]{\textsc{#1}}

\numberwithin{equation}{section}
\numberwithin{figure}{section}
\theoremstyle{plain}
\newtheorem{thm}{\protect\theoremname}
  \theoremstyle{definition}
  \newtheorem{defn}[thm]{\protect\definitionname}
  \theoremstyle{plain}
  \newtheorem{prop}[thm]{\protect\propositionname}
  \theoremstyle{definition}
  \newtheorem{example}[thm]{\protect\examplename}
  \theoremstyle{remark}
  \newtheorem{rem}[thm]{\protect\remarkname}
  \theoremstyle{plain}
  \newtheorem{lem}[thm]{\protect\lemmaname}

\usepackage{hyperref}
\hypersetup{
    colorlinks,
    citecolor=blue,
    filecolor=blue,
    linkcolor=blue,
    urlcolor=blue
}

\makeatother

\usepackage{babel}
\makeatletter
\addto\extrasfrench{%
   \providecommand{\fg}{\ifdim\lastskip>\z@\unskip\fi~\frqq}%
}

\makeatother
  \addto\captionsamerican{\renewcommand{\definitionname}{Definition}}
  \addto\captionsamerican{\renewcommand{\examplename}{Example}}
  \addto\captionsamerican{\renewcommand{\lemmaname}{Lemma}}
  \addto\captionsamerican{\renewcommand{\propositionname}{Proposition}}
  \addto\captionsamerican{\renewcommand{\remarkname}{Remark}}
  \addto\captionsamerican{\renewcommand{\theoremname}{Theorem}}
  \addto\captionsenglish{\renewcommand{\definitionname}{Definition}}
  \addto\captionsenglish{\renewcommand{\examplename}{Example}}
  \addto\captionsenglish{\renewcommand{\lemmaname}{Lemma}}
  \addto\captionsenglish{\renewcommand{\propositionname}{Proposition}}
  \addto\captionsenglish{\renewcommand{\remarkname}{Remark}}
  \addto\captionsenglish{\renewcommand{\theoremname}{Theorem}}
  \addto\captionsfrench{\renewcommand{\definitionname}{Definition}}
  \addto\captionsfrench{\renewcommand{\examplename}{Exemple}}
  \addto\captionsfrench{\renewcommand{\lemmaname}{Lemme}}
  \addto\captionsfrench{\renewcommand{\propositionname}{Proposition}}
  \addto\captionsfrench{\renewcommand{\remarkname}{Remarque}}
  \addto\captionsfrench{\renewcommand{\theoremname}{Theoreme}}
  \providecommand{\definitionname}{Definition}
  \providecommand{\examplename}{Example}
  \providecommand{\lemmaname}{Lemma}
  \providecommand{\propositionname}{Proposition}
  \providecommand{\remarkname}{Remark}
\providecommand{\theoremname}{Theorem}

\begin{document}

\title{equivariantly uniformly rational varieties}

\author{Charlie Petitjean}

\address{Charlie Petitjean, Institut de Math\'ematiques de Bourgogne, Universit\'e
de Bourgogne, 9 Avenue Alain Savary, BP 47870, 21078 Dijon Cedex,
France}

\email{charlie.petitjean@u-bourgogne.fr}


\begin{abstract}
We introduce equivariant versions of uniform rationality: given an
algebraic group $G$, a $G$-variety is called $G$-uniformly rational
(resp. $G$-linearly uniformly rational) if every point has a $G$-invariant
open neighborhood equivariantly isomorphic to a $G$-invariant open
subset of the affine space endowed with a $G$-action (resp. linear
$G$-action). We establish a criterion for $\mathbb{G}_{m}$-uniform
rationality of smooth affine varieties equipped with hyperbolic $\mathbb{G}_{m}$-actions
with a unique fixed point, formulated in term of their Altmann-Hausen
presentation. We prove the $\mathbb{G}_{m}$-uniform rationality of
Koras-Russell threefolds of the first kind and we also give an example
of a non $\mathbb{G}_{m}$-uniformly rational but smooth rational
$\mathbb{G}_{m}$-threefold associated to pairs of plane rational
curves birationally non equivalent to a union of lines.
\end{abstract}

\maketitle

\section*{Introduction}

A \emph{uniformly rational} variety is a variety for which every point
has a Zariski open neighborhood isomorphic to an open subset of an
affine space. A uniformly rational variety is in particular a smooth
rational variety, but the converse is an open question \cite[p.885]{Gr}. 

In this article, we introduce stronger equivariant versions of this
notion, in which we require in addition that the open subsets are
stable under certain algebraic group actions. The main motivation
is that for such varieties uniform rationality, equivariant or not,
can essentially be reduced to rationality questions at the quotient
level. We construct examples of smooth rational but not equivariantly
uniformly rational varieties, the question of their uniform rationality
is still open. We also establish equivariant uniform rationality of
large families of affine threefolds.

We focus mainly on actions of algebraic tori $\mathbb{T}$. The \emph{complexity}
of a $\mathbb{T}$-action on a variety is the codimension of a general
orbit, in the case of a faithful action, the complexity is thus simply
$\mathrm{dim}(X)-\mathrm{dim}(\mathbb{T})$. Complexity zero corresponds
to toric varieties, which are well-known to be uniformly rational
when smooth. In fact they are even $\mathbb{T}$-linearly uniformly
rational in the sense of Definition \ref{definition equivariant}
below. The same conclusion holds for smooth rational $\mathbb{T}$-varieties
of complexity one by a result of \cite[Chapter 4]{Ke-Kn-Mu-S}. In
addition, by \cite[Theorem 5]{Ar-Pe-S=0000FC} any smooth complete
rational $\mathbb{T}$-variety of complexity one admits a covering
by finitely many open charts isomorphic to the affine space.

In this article, as a step toward the understanding of $\mathbb{T}$-varieties
of higher complexity, we study the situation of affine threefolds
equipped with \emph{hyperbolic $\mathbb{G}_{m}$-actions}. We use
the general description developed by Altmann, Hausen and S\"{u}ss (see
\cite{A-H,A-H-S=0000FC} ) in terms of pairs $(Y,\mathcal{D})$, where
$Y$ is a variety of dimension $\mathrm{dim}(X)-\mathrm{dim}(\mathbb{T})$
and $\mathcal{D}$ a so-called \emph{polyhedral divisor} on $Y$.
In our situation, $Y$ is a rational surface and our main result,
Theorem \ref{Theorem equivalence}, allows us to translate equivariant
uniform rationality into a question of birational geometry of curves
on rational surfaces. 

The article is organized as follows. In the first section we introduce
equivariant versions of uniform rationality and summarize A-H presentations
of affine $\mathbb{G}_{m}$-varieties. The second section explains
how to use these presentations for the study of uniform rationality
of these varieties. In the third section, we focus on families of
\emph{$\mathbb{G}_{m}$-rational threefolds}, we show, for example,
that all \emph{Koras-Russell threefolds of the first kind}, and certain
ones of the second kind (see \cite{K-R,Ka-K-ML-R}) are equivariantly
uniformly rational and therefore uniformly rational. In a fourth section
we find examples of smooth rational $\mathbb{G}_{m}$-threefolds including
other Koras-Russell threefolds which are not equivariantly uniformly
rational. It is not known if these varieties are uniformly rational,
without any group action. In the last section, we introduce a weaker
notion of equivariant uniform rationality and we illustrate differences
between all these notions.

The author would like to thank Karol Palka and J\'er\'emy Blanc for helpful
discussions concerning Subsection 3.2.

\newpage{}

\section{\label{sec:1} Preliminaries}

\noindent

\subsection{Basic examples of uniformly rational varieties}

Recall that a a variety of dimension $n$ is called uniformly rational
if every point has a Zariski open neighborhood isomorphic to an open
subset of $\mathbb{A}^{n}$. Some partial results are known, for instance
every smooth complete rational surface is uniformly rational. In fact
it follows from \cite{Bo-B=0000F6,Bod-Hau-S-Vi} that the blowup of
a uniformly rational variety along a smooth subvariety is again uniformly
rational. Since open subset of uniformly rational varieties are uniformly
rational, it follows that every open subsets of the blowup of a uniformly
rational variety along a smooth subvariety is again uniformly rational.
In particular this holds for \emph{affine modifications} of uniformly
rational varieties along smooth subvarieties.
\begin{defn}
\label{affine modification}\cite{Ka-Z,Du} Let $(X,D,Z)$ be a triple
consisting of a variety $X$, an effective Cartier divisor $D$ on
$X$ and a closed sub-scheme $Z$ with ideal sheaf $\mathcal{I}_{Z}\subset\mathcal{O}_{X}(-D)$.
The \emph{affine modification} of the variety $X$ along $D$ with
center $Z$ is the scheme $X'=\tilde{X}_{Z}\setminus D'$ where $D'$
is the proper transform of $D$ in the blow-up $\tilde{X}_{Z}\rightarrow X$
of $X$ along $Z$. 
\end{defn}
A particular type of affine modification is the \emph{hyperbolic modification}
of a variety $X$ with center at a closed sub-scheme $Z\subset X$
(see \cite{Z}): It is defined as the affine modification of $X\times\mathbb{A}^{1}$
with center $Z\times\{0\}$, and divisor $X\times\{0\}$. As an immediate
corollary of \cite[Proposition 2.6]{Bo-B=0000F6}, we obtain the following
result:
\begin{prop}
\label{cor:The-affine-modifications}Affine modifications and hyperbolic
modifications of uniformly rational varieties along smooth centers
are again uniformly rational.
\end{prop}
\begin{example}
Let $\mathbb{A}^{n}=\mathrm{Spec}(\mathbb{C}[x_{1},\ldots,x_{n}])$
and let $I=(f,g)$ be the defining ideal of a smooth subvariety in
$\mathbb{A}^{n}$. Then the affine modification of $\mathbb{A}^{n}$
with center $I=(f,g)$ and divisor $D=\{f=0\}$ is isomorphic to the
subvariety $X'\subset\mathbb{A}^{n+1}$ defined by the equation: 
\[
\{g(x_{1},\ldots,x_{n})-yf(x_{1},\ldots,x_{n})=0\}\subset\mathbb{A}^{n+1}=\mathrm{Spec}(\mathbb{C}[x_{1},\ldots,x_{n},y]).
\]
 It is a uniformly rational variety.
\end{example}

\subsection{Equivariantly uniformly rational varieties}

Let $G$ be an affine algebraic group and let $X$ be a $G$-variety,
that is an algebraic variety endowed with a $G$-action. We introduce
equivariant versions, of uniform rationality.
\begin{defn}
\label{definition equivariant}Let $X$ be a $G$-variety and $x\in X$. 

i) We say that $X$ is $G$-\emph{linearly rational at the point $x$
}if there exists a $G$-stable open neighborhood $U_{x}$ of $x$,
a linear representation of $G\rightarrow GL_{n}(V)$ and a $G$-stable
open subset $U'\subset V\simeq\mathbb{A}^{n}$ such that $U_{x}$
is equivariantly isomorphic to $U'$.

ii) We say that $X$ is \emph{$G$-rational} \emph{at the point $x$}
if there exists an open $G$-stable neighborhood $U_{x}$ of $x$,
an action of $G$ on $\mathbb{A}^{n}$ and an open $G$-stable subset
$U'\subset\mathbb{A}^{n}$ such that $U_{x}$ is equivariantly isomorphic
to $U'$.

iii) A $G$-variety that is $G$-linearly rational (respectively $G$-rational)
at each point is called \emph{$G$-linearly uniformly rational} (respectively\emph{
$G$-uniformly rational})\emph{.}

iv) A $G$-variety that admits a unique fixed point $x_{0}$ by the
$G$-action is called \emph{$G$-linearly rational} (respectively
$G$-rational) if it is $G$-linearly rational (respectively $G$-rational)
at $x_{0}$.
\end{defn}
$G$-linearly uniformly rational or just $G$-uniformly rational varieties
are always uniformly rational. The converse is trivially false: for
instance the point $[1:0]$ in $\mathbb{P}^{1}$ does not admit any
$\mathbb{G}_{a}$-invariant affine open neighborhood for the action
defined by $t\cdot[u:v]\rightarrow[u+tv:v]$. 

For algebraic tori $\mathbb{T}$, as already mentioned in the introduction,
it is a classical fact that smooth toric varieties are $\mathbb{T}$-linearly
uniformly rational. Moreover, it is known that every effective $\mathbb{T}$-action
on $\mathbb{A}^{n}$ is linearisable for $\mathrm{dim}(\mathbb{T})\geq n-1$
(see \cite{Gu} for $n=2$ and \cite{B} for the general case), and
in another direction every algebraic $\mathbb{G}_{m}$-action on $\mathbb{A}^{3}$
is linearisable \cite{Ka-K-ML-R}. As a consequence, we obtain the
following:
\begin{thm}
\label{GLUR=00003DGUR}For $\mathbb{T}$-varieties of complexity $0$,
$1$ and for $\mathbb{G}_{m}$-threefolds the properties of being
$\mathbb{T}$-linearly uniformly rational and $\mathbb{T}$-uniformly
rational are equivalent.
\end{thm}

\subsection{\label{sec:2}Hyperbolic $\mathbb{G}_{m}$-actions on smooth varieties}

By a Theorem of Sumihiro (see \cite{Su}) every normal $\mathbb{G}_{m}$-variety
$X$ admits a cover by affine $\mathbb{G}_{m}$-stable open subsets.
This reduces the study of $\mathbb{G}_{m}$-linearly uniformly rational
varieties to the affine case. Recall that the coordinate ring $A$
of an affine $\mathbb{G}_{m}$-variety $X$ is $\mathbb{Z}$-graded
in a natural way by its subspaces $A_{n}:=\left\{ f\in A/f(\lambda\cdot x)=\lambda^{n}f(x),\forall\lambda\in\mathbb{G}_{m}\right\} $
of semi-invariants of weight $n$. In particular $A_{0}$ is the ring
of invariant functions on $X$. If $X$ is smooth with positively
graded coordinate ring, then by \cite{Kam-R}, $X$ has the structure
of a vector bundle over its fixed point locus $X^{\mathbb{G}_{m}}$,
and hence the question whether $X$ is $\mathbb{G}_{m}$-linearly
uniformly rational becomes intimately related to the uniform rationaly
of $X^{\mathbb{G}_{m}}$. In this subsection, we consequently focus
on \emph{hyperbolic} $\mathbb{G}_{m}$-actions. We summarize the correspondence
between smooth affine varieties $X$ endowed with an effective hyperbolic
$\mathbb{G}_{m}$-action and pairs $(Y,\mathcal{D})$ where $Y$ is
a variety, that we call \emph{A-H quotient}, and $\mathcal{D}$ is
a so-called \emph{segmental divisor} on $Y$. All the definitions
and constructions are adapted from \cite{A-H}. 
\begin{defn}
A $\mathbb{G}_{m}$-action is said to be \emph{hyperbolic} if there
is at least one $n_{1}<0$ and one $n_{2}>0$ such that $A_{n_{1}}$
and $A_{n_{2}}$ are nonzero. 
\end{defn}
\noindent
\begin{defn}
\label{The-A-H-quotient} Let $X=\mathrm{Spec}(A)$ be a smooth affine
variety equipped with a hyperbolic $\mathbb{G}_{m}$-action.

i) We denote by $q:X\rightarrow Y_{0}(X):=X/\!/\mathbb{G}_{m}=\mathrm{Spec}(A_{0})$
the \emph{categorical quotient} of $X$.

ii) The A-H quotient $Y(X)$ of $X$ is the blow-up $\pi:Y(X)\rightarrow Y_{0}(X)$
of $Y_{0}(X)$ with center at the closed subscheme defined by the
ideal $\mathcal{I}=\left\langle A_{d}\cdot A_{-d}\right\rangle $,
where $d>0$ is chosen such that $\bigoplus_{n\in\mathbb{Z}}A_{dn}$
is generated by $A_{0}$ and $A_{\pm d}$. It is a normal \emph{semi-projective}
variety (see \cite{A-H}). By virtue of \cite[Theorem 1.9,  proposition 1.4]{T},
$Y(X)$ is isomorphic to the fiber product of the schemes $Y_{\pm}(X)=\mathrm{Proj}_{A_{0}}(\bigoplus_{n\in\mathbb{Z}_{\geq0}}A_{\pm n})$
over $Y_{0}(X)$.
\end{defn}
In the remainder of the article, we use the notation $\pi:\tilde{V}_{I}\rightarrow V$
to refer to the blow-up of an affine variety $V$ with center at the
closed sub-scheme defined by the ideal $I\subset\Gamma(V,\mathcal{O}_{V})$. 
\begin{defn}
A \emph{segmental divisor} $\mathcal{D}$ on a normal algebraic variety
$Y$ is a formal finite sum $\mathcal{D}=\sum[a_{i},b_{i}]\otimes D_{i}$,
where $D_{i}$ are prime Weil divisors on $Y$ and $[a_{i},b_{i}]$
are closed intervals with rational bounds $a_{i}\leq b_{i}$.

The set of all closed intervals with rational bounds, admits a structure
of abelian semigroup for the \emph{Minkowski sum}, the Minkowski sum
of two intervals $[a_{i},b_{i}]$ and $[a_{j},b_{j}]$ being the interval
$[a_{i}+a_{j},b_{i}+b_{j}]$.
\end{defn}
Every element $n\in\mathbb{Z}$ determines a map from segmental divisors
to the group of Weil $\mathbb{Q}$-divisors on $Y$:

\[
\mathcal{D}=\sum[a_{i},b_{i}]\otimes D_{i}\rightarrow\mathcal{D}(n)=\sum q_{i}D_{i},
\]
 where for all $i$, $q_{i}\in\mathbb{Q}$ is the minimum of $na_{i}$
and $nb_{i}$. 
\begin{defn}
\label{A-porper-segmental-divisor,}A \emph{proper-segmental divisor
}(ps-divisor) $\mathcal{D}$ on a variety $Y$ is a segmental divisor
on $Y$ such that for every $n\in\mathbb{Z}$, $\mathcal{D}(n)$ satisfies
the following properties:

1) $\mathcal{D}(n)$ is a $\mathbb{Q}$-Cartier divisor on $Y$.

2) $\mathcal{D}(n)$ is semi-ample, that is, for some $p\in\mathbb{Z}_{>0}$,
$Y$ is covered by complements of supports of effective divisors linearly
equivalent to $\mathcal{D}(pn)$.

3) $\mathcal{D}(n)$ is big, that is, for some $p\in\mathbb{Z}_{>0}$,
there exists an effective divisor $D$ linearly equivalent to $\mathcal{D}(pn)$
such that $Y\setminus\mathrm{Supp}(D)$ is affine. 
\end{defn}
In the particular case of hyperbolic $\mathbb{G}_{m}$-action, the
main Theorem of \cite{A-H} can be reformulated as follows:
\begin{thm}
For any ps-divisor $\mathcal{D}$ on a normal semi-projective variety
$Y$ the scheme

\[
\mathbb{S}(Y,\mathcal{D})=\mathrm{Spec}(\bigoplus_{n\in\mathbb{Z}}\Gamma(Y,\mathcal{O}_{Y}(\mathcal{D}(n))))
\]
is a normal affine variety of dimension $\mathrm{dim}(Y)+1$ endowed
with an effective hyperbolic $\mathbb{G}_{m}$-action, whose A-H quotient
$Y(\mathbb{S}(Y,\mathcal{D}))$ is birationally isomorphic to $Y$.
Conversely any normal affine variety $X$ endowed with an effective
hyperbolic $\mathbb{G}_{m}$-action is isomorphic to $\mathbb{S}(Y(X),\mathcal{D})$
for a suitable ps-divisor $\mathcal{D}$ on $Y(X)$.
\end{thm}
\begin{rem}
Alternatively, see \cite{D,F-Z}, any finitely generated $\mathbb{Z}$-graded
algebra $A$ can be written in the form 

\[
A=\bigoplus_{n<0}\Gamma(Y,\mathcal{O}_{Y}(nD_{-}))\oplus\Gamma(Y,\mathcal{O}_{Y})\oplus\bigoplus_{n>0}\Gamma(Y,\mathcal{O}_{Y}(nD_{+}))
\]
where $(Y,D_{+},D_{-})$ is a triple consisting in a normal variety
$Y$ and suitable $\mathbb{Q}$-divisors $D_{+}$ and $D_{-}$ on
it. These two presentations are obtained from each other by setting
$D_{-}=\mathcal{D}(-1)$, $D_{+}=\mathcal{D}(1)$ and conversely $\mathcal{D}=\{1\}D_{+}+[0,1](-D_{-}-D_{+})$.
\end{rem}
\noindent
\begin{rem}
\label{ps-divisor construction}A method to determine a possible ps-divisor
$\mathcal{D}$ such that $X\simeq\mathbb{S}(Y,\mathcal{D})$ is to
embed $X$ as a $\mathbb{G}_{m}$-stable subvariety of an affine toric
variety (see \cite[section 11]{A-H}). The calculation is then reduced
to the toric case by considering an embedding in $\mathbb{A}^{n}$
endowed with a linear action of a torus $\mathbb{T}$ of sufficiently
large dimension $n$. The inclusion of $\mathbb{G}_{m}\hookrightarrow\mathbb{T}$
corresponds to an inclusion of the lattice $\mathbb{Z}$ of one parameter
subgroups of $\mathbb{G}_{m}$ in the lattice $\mathbb{Z}^{n}=N$
of one parameter subgroups of $\mathbb{T}$. We obtain the exact sequence:
\[
\xymatrix{0\ar[r] & \mathbb{Z}\ar[r]_{F} & N=\mathbb{Z}^{n}\ar[r]_{P}\ar@/_{1pc}/[l]_{s} & \overline{N}=\mathbb{Z}^{n}/\mathbb{Z}\ar[r] & 0}
,
\]
 where $F$ is given by the induced action of $\mathbb{G}_{m}$ on
$\mathbb{A}^{n}$ and $s$ is a section of $F$. Let $v_{i}$, for
$i=1,\ldots,n$, be the first integral vectors of the unidimensional
cones generated by the i-th column vectors of $P$ considered as rays
in the lattice $\overline{N}\simeq\mathbb{Z}^{n-1}$. Let $Z$ be
the toric variety of dimension $\mathrm{dim}(\mathbb{A}^{n})-\mathrm{dim}(\mathbb{T})$,
determined by the fan in $\overline{N}$ whose cones are generated
the $v_{i}$ for $i=1,\ldots,n$. Then each $v_{i}$ corresponds to
a $\mathbb{T}'$-invariant divisor where $\mathbb{T}'=\mathrm{Spec}(\mathbb{C}[\overline{N}^{\vee}]$.
By \cite[section 11]{A-H} $Z$ contains the A-H quotient of $X$
as a closed subset, and the support of $D_{i}$ is obtained by restricting
the $\mathbb{T}'$-invariant divisor corresponding to $v_{i}$ to
$Y$. If $X$ is the affine space endowed with a linear action of
$\mathbb{G}_{m}$, then $Z$ is itself the A-H quotient of $\mathbb{A}^{n}$.
The segment associated to the divisor $D_{i}$ is equal to $s(\mathbb{R}_{\geq0}^{n}\cap P^{-1}(v_{i}))$.
The section $s$ can further be chosen so that the number of non zero
coefficients in the associated matrix is minimal. The ps-divisor $\mathcal{D}$
from a such section will be called \emph{minimal}. We would like to
point out that this notion is more restrictive than that given in
\cite{A-H}, in particular every minimal ps-divisor in our sense is
also in the sense of \cite{A-H}.
\end{rem}

\section{Algebro-combinatorial criteria for $\mathbb{G}_{m}$-linear rationality}

Given a a smooth rational variety $X$ endowed with a hyperbolic $\mathbb{G}_{m}$-action
which admits a unique fixed point $x_{0}$, we develop in this section
a method to test whether $X$ is $\mathbb{G}_{m}$-rational. 
\begin{defn}
\cite[Definition 8.3]{A-H} Let $Y$ and $Y'$ be normal semi-projective
varieties and let $\mathcal{D}'=\sum[a'_{i},b'_{i}]\otimes D'_{i}$
and $\mathcal{D}=\sum[a{}_{i},b{}_{i}]\otimes D_{i}$ be ps-divisors
on $Y'$ and $Y$ respectively.

i) Let $\varphi:Y\rightarrow Y'$ be a morphism such that $\varphi(Y)$
is not contained in $\mathrm{Supp}(D'_{i})$ for any $i$. The \emph{polyhedral
pull-back} of $\mathcal{D}'$ is defined by $\varphi^{*}(\mathcal{D}'):=\sum[a'_{i},b'_{i}]\otimes\varphi^{*}(D'_{i})$,
where $\varphi^{*}(D'_{i})$ is the usual pull-back of $D'_{i}$. 

ii) Let $\varphi:Y\rightarrow Y'$ be a proper dominant map. The \emph{polyhedral
push-forward of $\mathcal{D}$ }is defined by $\varphi_{*}(\mathcal{D}):=\sum[a{}_{i},b{}_{i}]\otimes\varphi_{*}(D{}_{i})$,
where $\varphi_{*}(D{}_{i})$ is the usual push-forward of $D_{i}$.
\end{defn}
\noindent

Let $\varphi:Y\rightarrow Y'$ be a birational morphism and let $D'$
be a divisor on $Y'$, then we decompose the pull-back of $D'$ by
$\varphi$ as follows: $\varphi^{*}(D')=(\varphi^{-1})_{*}(D')+R$
where $(\varphi^{-1})_{*}(D')$ is the strict transform of $D'$ and
$R$ is supported in the exceptional locus of $\varphi$.
\begin{defn}
\label{pars equivalence}Two pairs $(Y_{i},D_{i})$, $i=1,2$ consisting
of a variety $Y_{i}$ and a Cartier divisor $D_{i}$ on $Y_{i}$ are
called \emph{birationally equivalent }if there exist a variety $Z$,
and two proper birational morphisms $\varphi_{i}:Z\rightarrow Y_{i}$
such that the strict transforms $(\varphi_{1}^{-1})_{*}(D_{1})$ and
$(\varphi_{2}^{-1})_{*}(D_{2})$ of $D_{1}$ and $D_{2}$ respectively
are equal. For ps-divisors, we extend this notion in the natural way
to pairs $(Y_{i},\mathcal{D}_{i})$ consisting of a semi-projective
variety $Y_{i}$ and a ps-divisor $\mathcal{D}_{i}$ on $Y_{i}$ using
the polyhedral pull-back defined above.
\end{defn}
\noindent

Since we consider hyperbolic $\mathbb{G}_{m}$-actions with a unique
fixed point, the construction of the A-H quotient $Y$ as in Definition
\ref{The-A-H-quotient} ensures that $Y(X)$ has only one exceptional
divisor $E$ over $Y_{0}(X)$. We denote by $\hat{\mathcal{D}}$ the
segmental divisor obtain from the ps-divisor $\mathcal{D}$ corresponding
to $X$ by removing all irreducible components whose support does
not intersect $E$. The following example illustrate a situation for
which $\hat{\mathcal{D}}\neq\mathcal{D}$. 
\begin{example}
Let $S$ be the affine surface defined by $\{x^{2}y+x=z^{2}\}\subset\mathbb{A}^{3}=\mathrm{Spec}(\mathbb{C}[x,y,z])$
and let $X:=S\times\mathbb{A}^{1}$ be the cylinder over $S$, endowed
with the hyperbolic $\mathbb{G}_{m}$-action induced by the linear
one $\lambda(x,y,z,t)\rightarrow(\lambda^{6}x,\lambda^{-6}y,\lambda^{3}z,\lambda^{2}t)$
on $\mathbb{A}^{4}=\mathrm{Spec}(\mathbb{C}[x,y,z,t])$. Using the
method described in Remark \ref{ps-divisor construction}, we find
that $X$ is equivariantly isomorphic to $\mathbb{S}(\tilde{\mathbb{A}}_{(u,v)}^{2},\mathcal{D})$
with:
\end{example}
\[
\mathcal{D}=\left\{ \frac{1}{2}\right\} D_{1}+\left\{ \frac{1}{2}\right\} D_{2}-\left\{ \frac{1}{3}\right\} D_{3}+\left[0,\frac{1}{6}\right]E,
\]

\begin{flushleft}
where $E$ is the exceptional divisor of the blow-up $\pi:\tilde{\mathbb{A}}_{(u,v)}^{2}\rightarrow\mathbb{A}^{2}\simeq\mathrm{Spec}(\mathbb{C}[u,v])\simeq\mathrm{Spec}(\mathbb{C}[yt^{3},yx])$
, and where $D_{1}$, $D_{2}$ and $D_{3}$ are the strict transforms
of the curves $L_{1}=\left\{ v=0\right\} $, $L_{2}=\left\{ 1+v=0\right\} $
$L_{3}=\left\{ u=0\right\} $ in $\mathbb{A}^{2}=\mathrm{Spec}(\mathbb{C}[u,v])$.
The divisor $D_{2}$ does not intersect the exceptional divisor $E$,
so: 
\[
\hat{\mathcal{D}}=\left\{ \frac{1}{2}\right\} D_{1}-\left\{ \frac{1}{3}\right\} D_{3}+\left[0,\frac{1}{6}\right]E.
\]
\par\end{flushleft}
\begin{thm}
\label{Theorem equivalence}Let $X$ be a smooth affine rational variety
endowed with a hyperbolic $\mathbb{G}_{m}$-action with a unique fixed
point $x_{0}$. Then $X$ is $\mathbb{G}_{m}$-rational if and only
if the following holds:

1) There exists pairs $(Y,\mathcal{D})$ and $(Y',\mathcal{D}')$
such that $\mathbb{S}(Y,\mathcal{D})$ is equivariantly isomorphic
to $X$ and $\mathbb{S}(Y',\mathcal{D}')$ is equivariantly isomorphic
to $\mathbb{A}^{n}$ endowed with a hyperbolic $\mathbb{G}_{m}$-action.

2) The pairs $(Y,\hat{\mathcal{D}})$ and $(Y',\hat{\mathcal{D}'})$
are birationally equivalent.
\end{thm}
\begin{proof}
Suppose that $X$ is $\mathbb{G}_{m}$-rational so that there exists
an open $\mathbb{G}_{m}$-stable neighborhood $U_{x_{0}}$ of $x_{0}$,
an action of $\mathbb{G}_{m}$ on $\mathbb{A}^{n}$, an open $\mathbb{G}_{m}$-stable
subvariety $U'\subset\mathbb{A}^{n}$, and an equivariant isomorphism
$\varphi:U_{x_{0}}\rightarrow U'$. We can always reduce to the case
where $U_{x_{0}}$ and $U'$ are principal open sets. Indeed $U_{x_{0}}$
is the complement of a closed stable subvariety of $X$ determined
by an ideal $\mathcal{I}=(f_{0},\ldots,f_{k})$ where each $f_{i}\in\Gamma(X,\mathcal{O}_{X})$
is semi-invariant. As $U_{x_{0}}$ contains $x_{0}$, at least one
of the $f_{i}$ does not vanish at $x_{0}$. Denoting this function
by $f$, the principal open subset $X_{f}:=X\setminus V(f)$ is contained
in $U_{x_{0}}$. The restriction of $\varphi$ to $X_{f}$ induces
an isomorphism between $X_{f}$ and $\varphi(X_{f})$. This yields
a divisor $U'\setminus\varphi(X_{f})$ on $U'$. Since $\mathbb{A}^{n}$
is factorial, this divisor is the restriction of a principal divisor
$\mathrm{Div}(f')$ on $\mathbb{A}^{n}$ for a certain regular function
$f'$. By construction $\varphi$ induces an equivariant isomorphism
between $X_{f}$ and $\mathbb{A}_{f'}^{n}$.

Note that any non-constant semi-invariant function $f\in\Gamma(X,\mathcal{O}_{X})$
such that $f(x_{0})\neq0$ is actually invariant. Indeed, letting
$w$ be the weight of $f$, we have $\lambda\cdot f(x_{0})=\lambda^{w}f(x_{0})=f(\lambda^{-1}\cdot x_{0})=f(x_{0})$
for all $\lambda\in\mathbb{G}_{m}$, and so $w=0$. 

Let $(Y,\mathcal{D})$ be the pair corresponding to $X$ with $\mathcal{D}$
minimal in the sense defined in Remark \ref{ps-divisor construction}.
We can identify every invariant function $f$ on $X$ non vanishing
at $x_{0}$ with an element $f$ of $\Gamma(Y,\mathcal{O}_{Y})$ such
that $V(f)\subset Y$ does not contain any irreducible component of
$\mathrm{Supp}(\hat{\mathcal{D}})$. Indeed, every such invariant
function corresponds via the algebraic quotient morphism $q:X\rightarrow Y_{0}$
to a function on $\Gamma(Y_{0},\mathcal{O}_{Y_{0}})$ which does not
vanish at $q(x_{0})$. Since the center of the blow-up $\pi:Y\rightarrow Y_{0}$
is supported by $q(x_{0})$ we can in turn identify $f$ with a regular
function on $Y$. We denote by $Y_{f}$ the corresponding open subset
of $Y$ where $f$ does not vanish, so that with our assumption $Y(X_{f})=Y_{f}$. 

By virtue of \cite[proposition 3.3]{A-H-S=0000FC}, for a ps-divisor
$\mathcal{D}$ on a normal, semiprojective variety $Y$ , let $\mathcal{D}_{f}$
be the localization of $\mathcal{D}$ by $f$. Then $\mathcal{D}_{f}$
is a ps-divisor on $Y_{f}$, and the canonical map $\mathcal{D}_{f}\rightarrow D$
describes the open embedding $X_{f}\rightarrow X$.

Denoting $i:Y_{f}\hookrightarrow Y$ the canonical open embedding,
we said that the pair $(Y_{f},\mathcal{D}_{f}=i^{*}(\mathcal{D}))$
describes the equivariant open embedding $j:X_{f}\simeq\mathbb{S}(Y_{f},i^{*}(\mathcal{D}))\hookrightarrow X$,
and we have the following diagram:

\[
\xymatrix{X_{f}\ar@{^{(}->}[rr]^{j}\ar[d]_{q} &  & X=\mathbb{S}(Y,\mathcal{D})\ar[d]_{q}\\
X_{f}/\!/\mathbb{G}_{m}=Y{}_{0,f}\ar@{^{(}->}[rr] &  & Y_{0}=X/\!/\mathbb{G}_{m}\\
Y_{f}\ar@{^{(}->}[rr]^{i}\ar[u]_{\pi_{\mid Y'}} &  & Y=Bl_{I}(Y_{0}).\ar[u]_{\pi}
}
\]

\begin{flushleft}
A similar description holds for the principal invariant open subset
$\mathbb{A}_{f'}^{n}$ of $\mathbb{A}^{n}$ endowed with a hyperbolic
$\mathbb{G}_{m}$-action. We denote the A-H quotient $Y(\mathbb{A}_{f'}^{n})$
of $\mathbb{A}_{f'}^{n}$ simply by $Y'_{f'}$ and the corresponding
ps-divisor by $\mathcal{D}'_{f'}$.
\par\end{flushleft}
\begin{flushleft}
By \cite[Corollary 8.12.]{A-H} $X_{f}$ and $\mathbb{A}_{f'}^{n}$
are equivariantly isomorphic if and only if there exist a normal semi-projective
variety $Y''$, birational morphisms $\sigma_{1}:Y{}_{f}\rightarrow Y''$
and $\sigma_{2}:Y'_{f'}\rightarrow Y''$ and a ps-divisor $\mathcal{D}''$
on $Y''$such that $\mathcal{D}\cong\sigma_{1}^{*}(\mathcal{D}'')$
and $\mathcal{D}'_{f'}\cong\sigma_{2}^{*}(\mathcal{D}'')$. Since
$\sigma_{1}$ is projective and birational, it either contracts the
unique exceptional divisor $E$ of $Y_{f}$ over $Y_{0,f}$, or it
is an isomorphism. But if $\sigma_{1}$ contracts $E$ then $\mathbb{S}(Y'',\mathcal{D}'')$
is not equivariantly isomorphic to $X_{f}$ by Definition \ref{The-A-H-quotient}.
Therefore $\sigma_{1}$ is an isomorphism. The same holds for $\sigma_{2}$. 
\par\end{flushleft}
\begin{flushleft}
Since $\mathcal{D}_{f}$ and $\mathcal{D}'_{f'}$ are minimal, the
pairs $(Y_{f},\mathcal{D}_{f})$ and $(Y'_{f'},\mathcal{D}'_{f'})$
are equivalent, that is, there exists an isomorphism $\Phi:Y_{f}\rightarrow Y'_{f'}$
such that $(\Phi^{-1})_{*}(\mathcal{D}'_{f'})=\mathcal{D}_{f}$. This
implies that the pairs $(Y,\hat{\mathcal{D}})$ and $(Y',\hat{\mathcal{D}'})$
are birationally equivalent and we obtain a commutative diagram :
\par\end{flushleft}
\[
\xymatrix{\mathbb{S}(Y',\mathcal{D}')=\mathbb{A}^{n}\ar[d]_{q} &  & \mathbb{A}_{f'}^{n}\simeq X_{f}\ar@{^{(}->}[rr]^{j}\ar[d]_{q}\ar@{_{(}->}[ll]_{j'} &  & X=\mathbb{S}(Y,\mathcal{D})\ar[d]_{q}\\
\mathbb{A}^{n}/\!/\mathbb{G}_{m}=Y'_{0} &  & Y'_{0,f'}\simeq Y{}_{0,f}\ar@{^{(}->}[rr]\ar@{_{(}->}[ll] &  & X/\!/\mathbb{G}_{m}=Y_{0}\\
Y'\ar[u] &  & Y'_{f'}\simeq Y_{f}\ar@{^{(}->}[rr]^{i}\ar[u]\ar@{_{(}->}[ll]_{i'} &  & Y.\ar[u]
}
\]

Conversely, assume that $X=\mathbb{S}(Y,\mathcal{D})$ and $\mathbb{A}^{n}=\mathbb{S}(Y',\mathcal{D}')$
endowed with an hyperbolic $\mathbb{G}_{m}$-action are such that
the pairs $(Y,\hat{\mathcal{D}})$ and $(Y',\hat{\mathcal{D}'})$
are birationally equivalent. We can further assume that there exists
a birational map between $Y$ and $Y'$ which restricts to an isomorphism
$\phi:Y_{g}\rightarrow Y'_{g'}$ between the principal open sets $Y_{g}$
of $Y$ and $Y'_{g'}$ of $Y'$ corresponding to suitable functions
$g\in\varGamma(Y,\mathcal{O}_{Y})$ and $g'\in\varGamma(Y',\mathcal{O}_{Y'})$
whose zero loci do not intersect the exceptional divisors of $Y\rightarrow Y_{\text{0 }}$
and $Y'\rightarrow Y'_{0}$ respectively. Similarly as above the function
$g$ can be identified with an invariant function on $X$ which does
not vanish at $x_{0}$. By virtue of \cite[Proposition 3.3]{A-H-S=0000FC}
the pair $(Y_{g},\mathcal{D}_{g})$ describes the equivariant open
embedding $X_{g}\simeq\mathbb{S}(Y_{g},\mathcal{D}_{g})\hookrightarrow X$.
In the same way, $g'$ corresponds to an invariant function on $\mathbb{A}^{n}$and
the pair $(Y_{g'},\mathcal{D}_{g'})$ describes the equivariant open
embedding $\mathbb{A}_{g'}^{n}\simeq\mathbb{S}(Y'_{g'},\mathcal{D}'_{g'})\hookrightarrow\mathbb{A}^{n}$.
This gives the result.
\end{proof}

\section{\label{sec:3}Examples of $\mathbb{G}_{m}$-uniformly rational threefolds}

In the particular case of affine threefolds, $\mathbb{G}_{m}$-linear
uniform rationality is reduced (by the previous section) to a problem
of birational geometry in dimension $2$. Indeed, using Theorem \ref{Theorem equivalence},
the question may then be considered at the level of the A-H quotients
which are rational semi-projective surfaces.

\noindent

\subsection{Hyperbolic $\mathbb{G}_{m}$-action on $\mathbb{A}^{3}$}

Using this presentation and the fact that every algebraic $\mathbb{G}_{m}$-action
on $\mathbb{A}^{3}=\mathrm{Spec}(\mathbb{C}[x,y,z])$ is linearisable
\cite{Ka-K-ML-R}, we are able to characterize hyperbolic $\mathbb{G}_{m}$-actions
on $\mathbb{A}^{3}$ in terms of their A-H presentations. Indeed let
$\mathbb{G}_{m}\times\mathbb{A}^{3}\rightarrow\mathbb{A}^{3}$ be
an effective hyperbolic $\mathbb{G}_{m}$-action given by $\lambda\cdot(x,y,z)\rightarrow(\lambda^{a}x,\lambda^{b}y,\lambda^{-c}z)$
with $(a,b,c)\in\mathbb{Z}_{>0}^{3}$. After a suitable $\mathbb{G}_{m}$-invariant
cyclic cover along coordinate axes, we can assume that $\mathbb{A}^{3}/\!/\mathbb{G}_{m}\simeq\mathbb{A}^{2}$,
and the relation between such cyclic covers ans the A-H presentations
of the $\mathbb{T}$-varieties are controlled by \cite{P1}. Let $(\alpha,\beta,\gamma)\in\mathbb{Z}^{3}$
be such that $\alpha a+\beta b-\gamma c=1$. Let $\rho(a,c)$ be the
greatest common divisor of $a$ and $c$, let $\rho(b,c)$ be the
greatest common divisor of $b$ and $c$, and let $\delta$ be the
greatest common divisor of $\frac{a}{\rho(a,c)}$ and $\frac{b}{\rho(b,c)}$
Then we have:
\begin{prop}
\label{classification A3} Up to equivariant cyclic covers, every
$\mathbb{A}^{3}$ endowed with a hyperbolic $\mathbb{G}_{m}$-action
is equivariantly isomorphic to the $\mathbb{G}_{m}$-variety $\mathbb{S}(Y,\mathcal{D})$
with $Y$ and $\mathcal{D}$  defined as follows: 

i) $Y$ is isomorphic to a blow-up $\pi:\tilde{\mathbb{A}}^{2}\rightarrow\mathbb{A}^{2}$
of $\mathbb{A}^{2}$ at the origin.

ii) $\mathcal{D}$ is of the form: 
\[
\mathcal{D}=\left\{ \frac{\alpha\rho(a,c)}{c}\right\} \otimes D_{1}+\left\{ \frac{\beta\rho(b,c)}{c}\right\} \otimes D_{2}+\left[\frac{\gamma}{\delta},\frac{\gamma}{\delta}+\frac{1}{\delta c}\right]\otimes E,
\]
 with $D_{1}$, $D_{2}$ are strict transforms of the coordinate axes
and $E$ is the exceptional divisor of $\pi$.
\end{prop}
\begin{proof}
Let $\mathbb{A}^{3}$ be endowed with a linear action of $\mathbb{G}_{m}$,
the A-H presentation is obtained from the following exact sequence:

\[
\xymatrix{0\ar[r] & \mathbb{Z}\ar[r]_{F} & \mathbb{Z}^{3}\ar[r]_{P}\ar@/_{1pc}/[l]_{s} & \mathbb{Z}^{2}\ar[r] & 0}
,
\]

\begin{flushleft}
by the method described in Remark \ref{ps-divisor construction},
where $F=\overset{}{}^{t}(a,b,-c)$, $s=(\alpha,\beta,\gamma)$ and
$P=\left(\begin{array}{ccc}
\frac{c}{\rho(a,c)} & 0 & \frac{a}{\rho(a,c)}\\
0 & \frac{c}{\rho(b,c)} & \frac{b}{\rho(b,c)}
\end{array}\right)$.
\par\end{flushleft}
\noindent \begin{flushleft}
The algebraic quotient of $\mathbb{A}^{3}$ for an hyperbolic $\mathbb{G}_{m}$-action
is isomorphic to $\mathbb{A}^{2}/\!/\mu$ where $\mu$ is a finite
cyclic group \cite{G-K-R}, thus the A-H quotient $Y(\mathbb{A}^{3})$
is by construction a blow-up of $\mathbb{A}^{2}/\!/\mu$. In this
case $Y(\mathbb{A}^{3})$ is smooth and corresponds to the toric variety
$Z$ defined in Remark \ref{ps-divisor construction} that is a blow-up
of $\mathbb{A}^{2}$ whose center is supported at the origin.
\par\end{flushleft}
\noindent \begin{flushleft}
Let now us consider each $v_{i}$, for $i=1,\ldots,3$, as in remark
\ref{ps-divisor construction}, that is the first integral vectors
of the unidimensional cones generated by the i-th column vectors of
$P=\left(\begin{array}{ccc}
\frac{c}{\rho(a,c)} & 0 & \frac{a}{\rho(a,c)}\\
0 & \frac{c}{\rho(b,c)} & \frac{b}{\rho(b,c)}
\end{array}\right)$. The first two $v_{1}=\left(\begin{array}{c}
1\\
0
\end{array}\right)$ and $v_{2}=\left(\begin{array}{c}
0\\
1
\end{array}\right)$ as rays defining a toric variety correspond to the generators of
$\mathbb{A}^{2}$ and thus the associated divisors are the strict
transforms of the coordinate axes and the last on $v_{3}$ corresponds
to the exceptional divisor. To determine the associated coefficients,
we used the formula $[a_{i},b{}_{i}]=s(\mathbb{R}_{\geq0}^{n}\cap P^{-1}(v_{i}))$
given in remark \ref{ps-divisor construction}.
\par\end{flushleft}
\end{proof}
\begin{example}
\cite[example 1.4.8]{P2} The presentation of $\mathbb{A}^{3}=\mathrm{Spec}(\mathbb{C}[x,y,z])$
equipped with the hyperbolic $\mathbb{G}_{m}$-action $\lambda\cdot(x,y,z)=(\lambda^{2}x,\lambda^{3}y,\lambda^{-6}z)$
is $\mathbb{S}(\tilde{\mathbb{A}}_{(u,v)}^{2},\mathcal{D})$ with
$\pi:\tilde{\mathbb{A}}_{(u,v)}^{2}\rightarrow\mathbb{A}^{2}$ the
blow-up of $\mathbb{A}^{2}=\mathrm{Spec}(\mathbb{C}[u,v])$ at the
origin and 
\[
\mathcal{D}=\left\{ -\frac{1}{3}\right\} D{}_{1}+\left\{ \frac{1}{2}\right\} D{}_{2}+\left[0,\frac{1}{6}\right]E,
\]
where $E$ is the exceptional divisor of the blow-up, $D_{1}$ and
$D_{2}$ are the strict transforms of the lines $\left\{ u=0\right\} $
and $\left\{ v=0\right\} $ in $\mathbb{A}^{2}$ respectively. Indeed,
$\mathbb{A}^{3}/\!/\mathbb{G}_{m}=\mathrm{Spec}(\mathbb{C}[u,v])$
and $d>0$ in Definition \ref{The-A-H-quotient} has to be chosen
so that $\bigoplus_{n\in\mathbb{Z}}A_{dn}$ is generated by $A_{0}$
and $A_{\pm d}$. This is the case if $d$ is the least common multiple
of the weights of the $\mathbb{G}_{m}$-action on $\mathbb{A}^{3}$.
Thus $d=6$ and $Y(X)$ is the blow-up of $\mathbb{A}^{2}=\mathrm{Spec}(\mathbb{C}[u,v])$
with center at the closed sub-scheme with ideal $(u,v)$, i.e. the
origin with our choice of coordinates.
\end{example}

\subsection{$\mathbb{G}_{m}$-linear uniform rationality.}

In this subsection we will prove that some hypersurfaces of $\mathbb{A}^{4}$
are $\mathbb{G}_{m}$-linearly uniformly rational. In particular,
every \emph{Koras-Russell threefold of the first kind} $X$ is $\mathbb{G}_{m}$-linearly
uniformly rational. These varieties are defined by equations of the
form: 
\[
\{x+x^{d}y+z^{\alpha_{2}}+t^{\alpha_{3}}=0\}\subset\mathbb{A}^{4}=\mathrm{Spec}(\mathbb{C}[x,y,z,t]),
\]
 where $d\geq2$, and $\alpha_{2}$ and $\alpha_{3}$ are coprime.
They are smooth rational and endowed with hyperbolic $\mathbb{G}_{m}$-actions
with algebraic quotients isomorphic to $\mathbb{A}^{2}/\!/\mu$ where
$\mu$ is a finite cyclic group. They have been classified by Koras
and Russell, in the context of the linearization problem for $\mathbb{G}_{m}$-actions
on $\mathbb{A}^{3}$ \cite{Ka-K-ML-R}.

These threefolds can be viewed  as affine modifications of $\mathbb{A}^{3}=\mathrm{Spec}(\mathbb{C}[x,z,t])$
along the principal divisor $D_{f}=\{f=0\}$ with center $I=(f,g)$
where $f=-x^{d}$ and $g=x+z^{\alpha_{^{2}}}+t^{\alpha_{3}}$. But
since the center is supported on the cuspidal curve included in the
plane $\{x=0\}$ and given by the equation : $C=\{x=z^{\alpha_{2}}+t^{\alpha_{3}}=0\}$
(see \cite{Z}) their uniformly rationality does not follow directly
from Corollary \ref{cor:The-affine-modifications}.

\noindent

\subsubsection{A general construction}

Here we give a general criterion to decide the $\mathbb{G}_{m}$-uniform
rationality of certain threefolds, arising as stable hypersurfaces
of $\mathbb{A}^{4}$ endowed with a linear $\mathbb{G}_{m}$-action.
Since $X$ is rational, its A-H quotient $Y(X)$ is also rational.

The aim is to use the notion of birational equivalence of ps-divisors
to construct an isomorphism between a $\mathbb{G}_{m}$-stable open
set of the variety $X$ with a corresponding stable open subset of
$\mathbb{A}^{3}$. By Theorem \ref{Theorem equivalence} and Proposition
\ref{classification A3}, the technique is to consider a well chosen
sequence of birational transformations $Y(X)\rightarrow\tilde{\mathbb{A}}^{2}$
which maps the support of the ps-divisor corresponding to the threefolds
$X$ on to the strict transforms of the coordinate lines and the exceptional
divisor in $\tilde{\mathbb{A}}^{2}$. Moreover, since we look for
$\mathbb{G}_{m}$-stable open subset of $X$ containing the fixed
point, it is enough to consider birational map $Y_{0}(X)\rightarrow\mathbb{A}^{2}$
which send a pair of curves on the coordinates axes of $\mathbb{A}^{2}$.

Let $p\in\mathbb{C}[v]$ be a polynomial of degree $k\geq1$ such
that $p(0)=0$, let $\alpha_{2}$, $\alpha_{3}$ and $d$ be integers
such that $d\alpha_{3}$ and $\alpha_{2}$ are coprime. Let $X$ be
a hypersurface in $\mathbb{A}^{4}=\mathrm{Spec}(\mathbb{C}[x,y,z,t])$
defined by one of the following equations:

\[
X=\{y^{d-1}z^{\alpha_{2}}+t^{\alpha_{3}}+p(xy)/y=0\}\subset\mathbb{A}^{4}=\mathrm{Spec}(\mathbb{C}[x,y,z,t]).
\]

Every such $X$ is endowed with a hyperbolic $\mathbb{G}_{m}$-action
induced by the linear action on \noun{$\mathbb{A}^{4}$} defined by
$\lambda\cdot(x,y,z,t)=(\lambda^{\alpha_{2}\alpha_{3}}x,\lambda^{-\alpha_{2}\alpha_{3}}y,\lambda^{d\alpha_{3}}z,\lambda^{\alpha_{2}}t)$.
The unique fixed point for this action is the origin of $\mathbb{A}^{4}$
and is a point of $X$.
\begin{thm}
\label{theoreme famiille unif} With the notation above we have:

1) $X$ is equivariantly isomorphic to $\mathbb{S}(\tilde{\mathbb{A}}_{(u,v^{d})}^{2},\mathcal{D})$
with 
\[
\mathcal{D}=\left\{ \frac{a}{\alpha_{2}}\right\} D{}_{1}+\left\{ \frac{b}{\alpha_{3}}\right\} D{}_{2}+\left[0,\frac{1}{\alpha_{2}\alpha_{3}}\right]E,
\]
where $E$ is the exceptional divisor of the blow-up $\pi:\tilde{\mathbb{A}}_{(u,v^{d})}^{2}\rightarrow\mathbb{A}^{2}=\mathrm{Spec}(\mathbb{C}[u,v])$,
$D_{1}$ and $D_{2}$ are the strict transforms of the curves $L_{1}=\left\{ u=0\right\} $
and $L_{2}=\left\{ u+p(v)=0\right\} $ in $\mathbb{A}^{2}$ respectively,
and $(a,b)\in\mathbb{Z}^{2}$ are chosen so that $ad\alpha_{3}+b\alpha_{2}=1$.

2) $X$ is smooth if and only if $L_{1}+L_{2}$ is a SNC divisor in
$\mathbb{A}^{2}$.

3) Under these conditions, $X$ is \emph{$\mathbb{G}_{m}$}-linearly
rational at $(0,0,0,0)$.
\end{thm}
\begin{proof}
1) The A-H presentation is obtained from the following exact sequence:

\[
\xymatrix{0\ar[r] & \mathbb{Z}\ar[r]_{F} & \mathbb{Z}^{4}\ar[r]_{P}\ar@/_{1pc}/[l]_{s} & \mathbb{Z}^{3}\ar[r] & 0}
,
\]

\begin{flushleft}
by the method described in Remark \ref{ps-divisor construction},
where $F=\overset{}{}^{t}(\alpha_{2}\alpha_{3},-\alpha_{2}\alpha_{3},d\alpha_{3},\alpha_{2})$,
$P=\left(\begin{array}{cccc}
1 & 1 & 0 & 0\\
0 & d & \alpha_{2} & 0\\
0 & 1 & 0 & \alpha_{3}
\end{array}\right)$ and $s=(0,0,a,b)$ chosen such that $ad\alpha_{3}+b\alpha_{2}=1$.
\par\end{flushleft}
The corresponding toric variety $Z$ is the blow-up of $\mathbb{A}^{3}=\mathrm{Spec}(\mathbb{C}[u,v,w])$
along the sub-scheme with defining by the ideal $I=(u,v^{d},v^{d-1}w,\ldots,vw^{d-1},w^{d})$,
and $Y$ corresponds to the strict transform by $\pi:\tilde{\mathbb{A}}_{I}^{3}\rightarrow\mathbb{A}^{3}\simeq\mathbb{A}^{4}/\!/\mathbb{G}_{m}$
of the surface $\left\{ u+w+p(v)=0\right\} \simeq\mathrm{Spec}(\mathbb{C}[u,v])$,
that is, $Y\simeq\tilde{\mathbb{A}}_{(u,v^{d})}^{2}$ (see \cite[section 3.1]{P1}).

The ps-divisor $\mathcal{D}$ is of the form $\left\{ \frac{a}{\alpha_{2}}\right\} D{}_{1}+\left\{ \frac{b}{\alpha_{3}}\right\} D{}_{2}+\left[0,\frac{1}{\alpha_{2}\alpha_{3}}\right]E$,
where $D_{1}$ corresponds to the restriction to $Y$ of the toric
divisor given by the ray $v_{3}$ and $D_{2}$ corresponds to the
restriction to $Y$ of the toric divisor given by the ray $v_{4}$,
that is, the strict transforms of the curves $\left\{ u=y^{d}z^{\alpha_{2}}=0\right\} $
and $\left\{ w=yt^{\alpha_{3}}=-u-p(v)=0\right\} $ in $\mathbb{A}^{2}$
respectively. The divisor $E$ corresponds to the divisor given by
$v_{2}$, that is, the exceptional divisor of $\pi:\tilde{\mathbb{A}}_{(u,v^{d})}^{2}\rightarrow\mathbb{A}^{2}$.

2) As $p(0)=0$, the equation of $X$ takes the form : 
\[
y^{d-1}z^{\alpha_{2}}+t^{\alpha_{3}}+x\prod_{i=1}^{k}(xy+\alpha_{i})=0,
\]

\begin{flushleft}
and using the Jacobian criterion, we conclude that $X$ is smooth
if and only if $\alpha_{i}\neq\alpha_{j}$ for $i\neq j$.
\par\end{flushleft}
\begin{flushleft}
3) Let $D=L_{1}+L_{2}\subset\mathbb{A}_{(u,v)}^{2}$ and let $\mathbb{A}_{(u,v)}^{2}\hookrightarrow\mathbb{P}_{[u:v:w]}^{2}$
be the embedding of $\mathbb{A}^{2}$ as the complement of the line
at the infinity $L_{\infty}=\{w=0\}$. We denote by $\bar{D}=\bar{L}_{1}+\bar{L}_{2}$
the closure of $D$ in $\mathbb{P}_{(u:v:w)}^{2}$ (see Figure 3.1).
The only singularity is at the intersection of $\bar{L}_{2}$ and
$L_{\infty}$. After a sequence of elementary birational transformations
we reach the $k$-th Hirzebruch surface $\mathbb{F}_{k}=\mathbb{P}(\mathcal{O}_{\mathbb{P}^{1}}\oplus\mathcal{O}_{\mathbb{P}^{1}}(k))$.
The proper transform of $\bar{L}_{2}$ is a smooth curve intersecting
the section of negative self-intersection transversally (see Figure
3.2). The second step is the blow-up of all the intersection points
between $\bar{L}_{1}$ and $\bar{L}_{2}$ except the point corresponding
to the origin in $\mathbb{A}^{2}$, followed by the contraction of
the proper transform of the fiber passing through each points of the
blowup (see Figure 3.3). The final configuration is then the Hirzebruch
surface $\mathbb{F}_{1}$ in which the proper transforms of $\bar{L}_{1}$
and $\bar{L}_{2}$ have self-intersection $1$ and intersect each
other in a unique point. Then $\mathbb{P}^{2}$ and the desired divisor
are obtained from $\mathbb{F}_{1}$ by contracting the negative section
(see Figure 3.4).
\par\end{flushleft}
\begin{figure}[H]
\begin{minipage}[b][1\totalheight][t]{0.45\columnwidth}%
\includegraphics[scale=0.45]{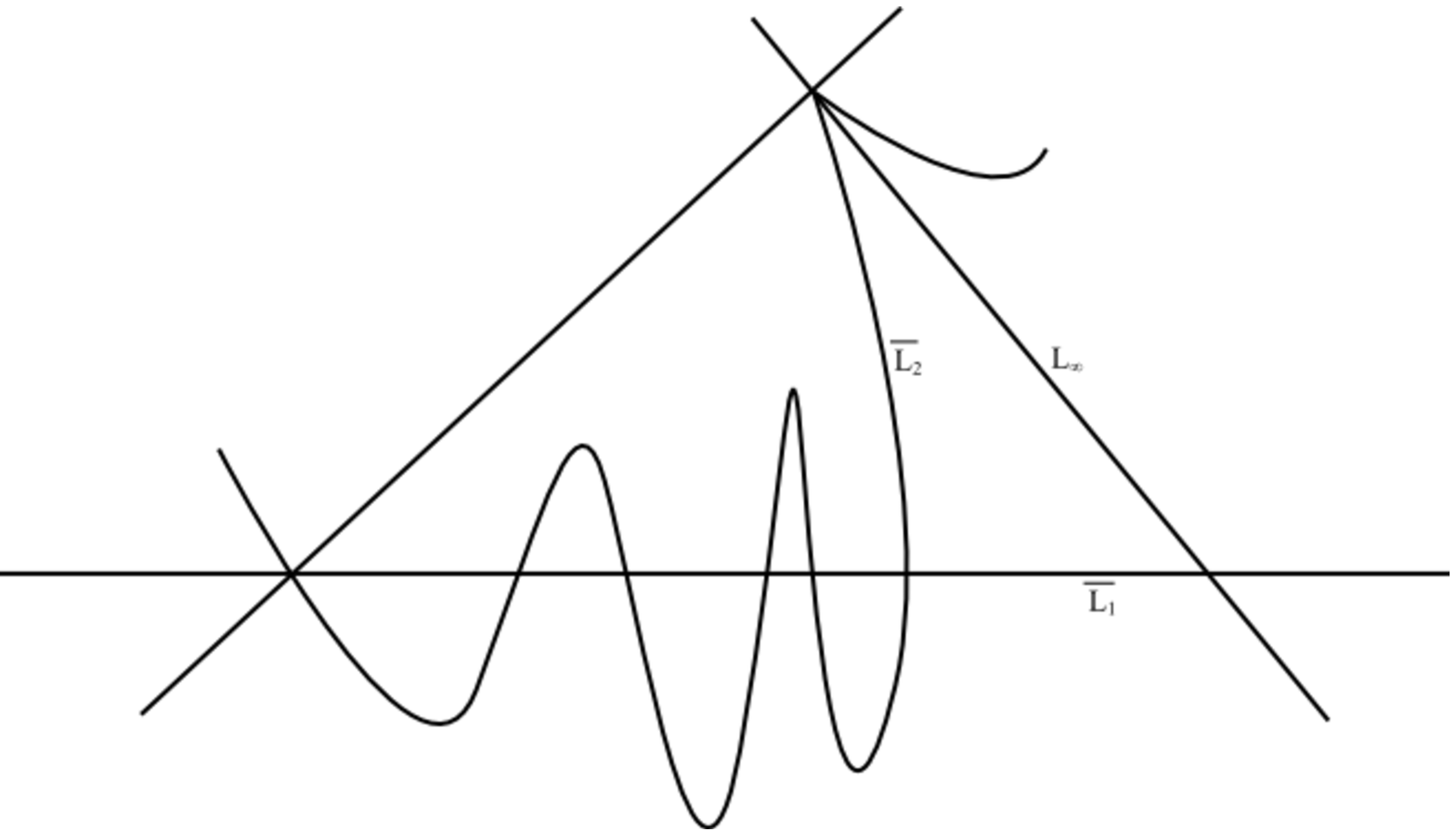}\caption{Embedding in $\mathbb{P}^{2}$ of the divisor in $\mathbb{P}^{2}$.}
\end{minipage}\hfill{}%
\begin{minipage}[b][1\totalheight][t]{0.45\columnwidth}%
\includegraphics[scale=0.45]{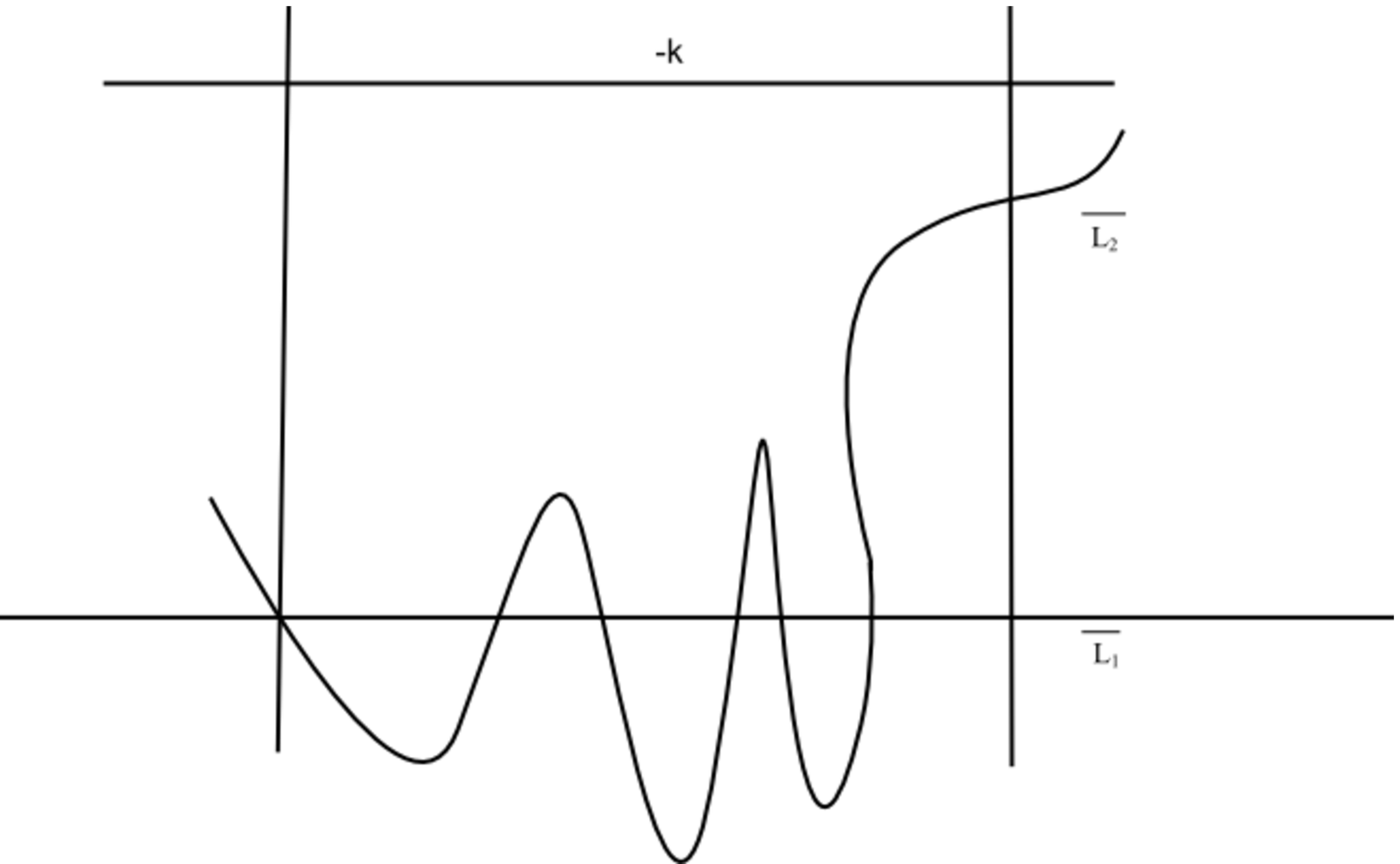}

\caption{First sequence of blow-ups and contractions, to obtain a smooth normal
crossing divisor divisor in $\mathbb{F}_{k}$.}
\end{minipage}
\end{figure}

\begin{figure}[H]
\begin{minipage}[b][1\totalheight][t]{0.45\columnwidth}%
\includegraphics[scale=0.45]{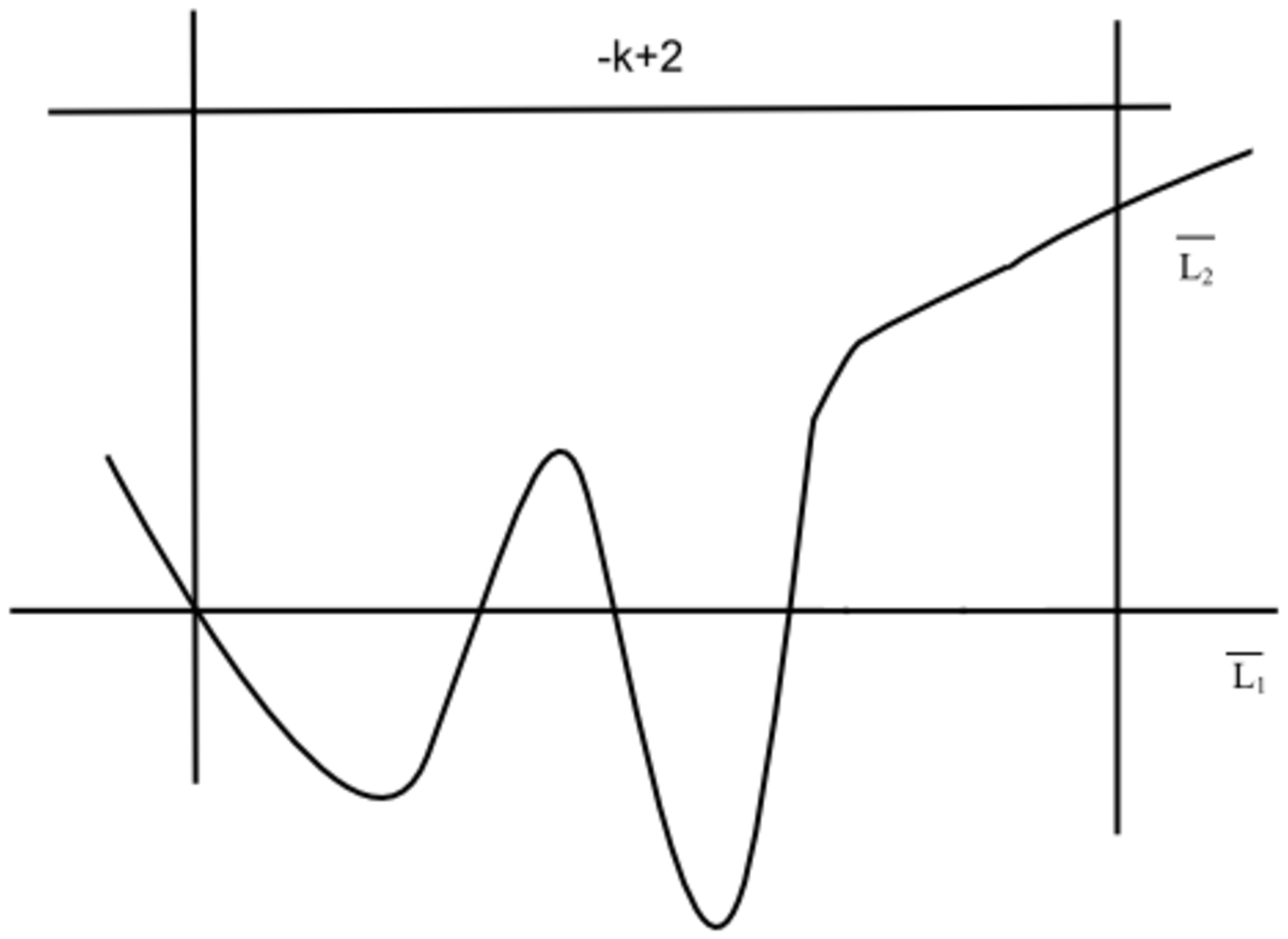}\caption{Intermediate step, resolution of the crossings, to obtain a divisor
in $\mathbb{F}_{k-2}$.}
\end{minipage}\hfill{}%
\begin{minipage}[b][1\totalheight][t]{0.45\columnwidth}%
\includegraphics[scale=0.45]{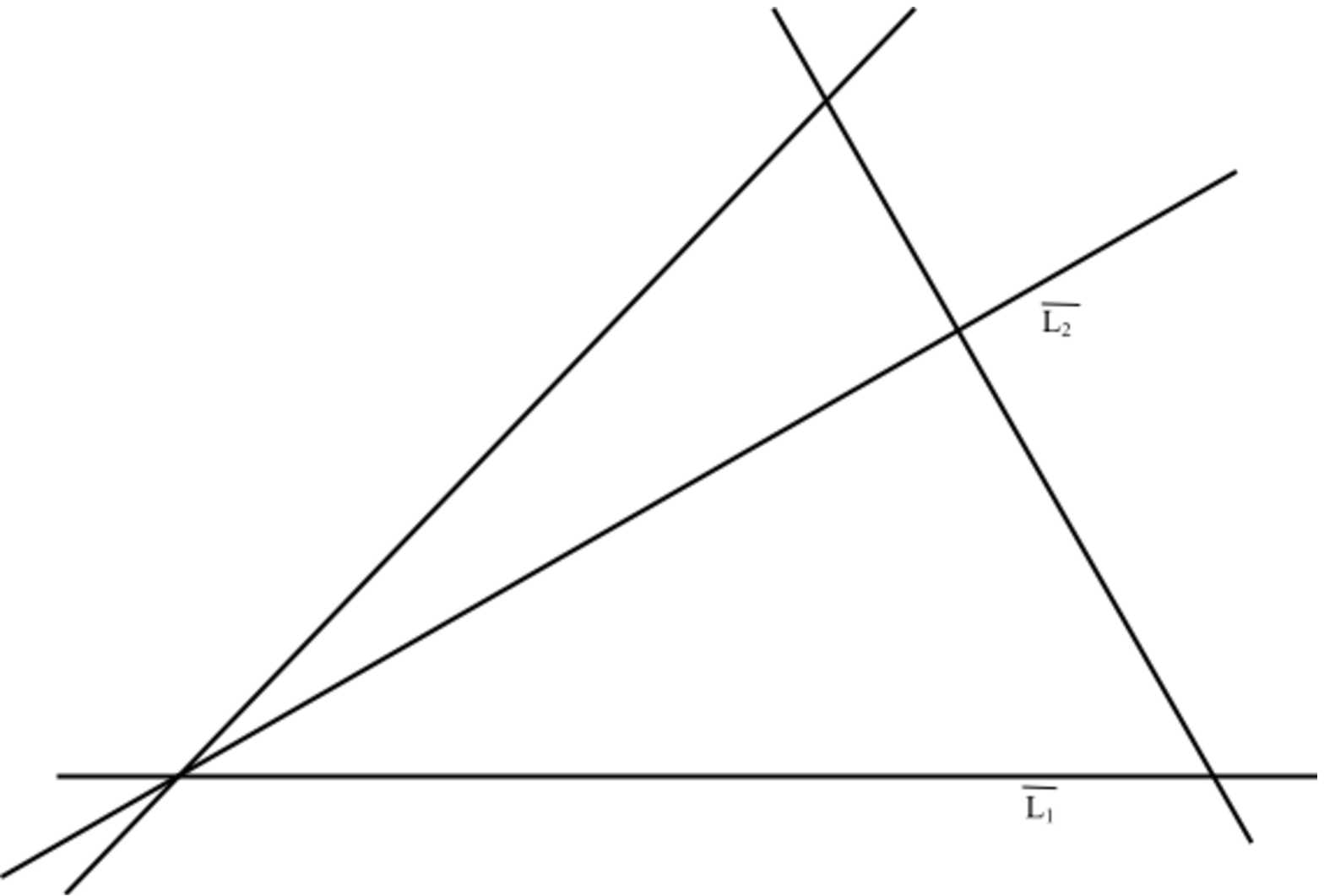}\caption{Final resolution, to obtain a divisor in $\mathbb{P}^{2}$}
\end{minipage}
\end{figure}

This resolution gives a birational map from the A-H quotient of $X$
to the A-H quotient of $\mathbb{A}^{3}$ which induces an isomorphism
in a neighborhood of the origin of $\mathbb{A}^{2}$. By Theorem \ref{Theorem equivalence}
this gives a $\mathbb{G}_{m}$-equivariant isomorphism between an
open neighborhood of the origin in $X$ and on open neighborhood of
the origin in $\mathbb{A}^{3}$. 

Let $p(v)=v(1+g(v))$ be the polynomial which appears in the statement,
and let $\phi$ be the birational map defined by: 
\[
\phi:(u,v)\rightarrow(-u'(g(v'+u')+1),v'+u').
\]
Its inverse is defined by 
\[
\phi^{-1}:(u',v')\rightarrow(-\frac{u}{1+g(v)},v+\frac{u}{1+g(v)}).
\]

\begin{flushleft}
Then $\phi(u+p(v))=v'(g(v'+u')+1)$ and we obtain :
\par\end{flushleft}
\[
\xymatrix{Y(\mathbb{A}^{n}) & ^{i} & Y'=\tilde{\mathbb{A}}_{(u,v^{d})}^{2}\setminus V(1+g(v))\simeq\tilde{\mathbb{A}}_{(u',v'^{d})}^{2}\setminus V(g(v'+u')+1)\ar@{^{(}->}[rr]\ar@{_{(}->}[ll] & ^{i'} & Y(X)}
\]

\begin{flushleft}
and $i:Y'\hookrightarrow\tilde{\mathbb{A}}_{(u,v^{d})}^{2}$, then
$\mathbb{S}(Y',i^{*}(\mathcal{D}))=U$ is an equivariant open neighborhood
of the fixed point in $X$, which is moreover equivariantly isomorphic
to an open subset of $\mathbb{A}^{3}=\mathrm{Spec}(\mathbb{C}[Y,Z,T])$
endowed of the hyperbolic $\mathbb{G}_{m}$-action. The action on
$\mathbb{A}^{3}$ is defined by $\lambda\cdot(Y,Z,T)=(\lambda^{-\alpha_{2}\alpha_{3}}Y,\lambda^{d\alpha_{3}}Z,\lambda^{\alpha_{2}}T)$
using Proposition \ref{classification A3}. 
\par\end{flushleft}

\end{proof}
\noindent
\begin{rem}
In the particular case where $L_{1}+L_{2}$ is not a smooth normal
crossing divisor in $\mathbb{A}^{2}$, that is the point 2 of the
Theorem \ref{theoreme famiille unif} is not satisfied, but the crossing
of $L_{1}$ and $L_{2}$ at the origin is transversal, then $\mathbb{S}(Y,\mathcal{D})$
is equivariantly isomorphic to a normal but not smooth $\mathbb{G}_{m}$-variety
$V$ with a unique fixed point contained in its regular locus. The
same process as before can be applied, and the variety $V$ admits
an open $\mathbb{G}_{m}$-stable neighborhood of the fixed point isomorphic
to a $\mathbb{G}_{m}$-stable neighborhood of the fixed point of $\mathbb{A}^{3}$
endowed with a linear hyperbolic $\mathbb{G}_{m}$-action.

In other words $V$ is $\mathbb{G}_{m}$-linearly rational, but not
uniformly rational, since it is singular.
\end{rem}

\subsubsection{Applications}

Specifying the coefficients of the polynomial $p\in\mathbb{C}[v]$
defined in the previous sub-section, we list below particular hypersurfaces
of $\mathbb{A}^{4}$ which are $\mathbb{G}_{m}$-uniformly rational.
\begin{prop}
\label{cover-1}The following hypersurfaces in $\mathbb{A}^{4}=\mathrm{Spec}(\mathbb{C}[x,y,z,t])$
are \emph{$\mathbb{G}_{m}$-}linearly rational:

\[
X_{1}=\{x+x^{k}y^{k-1}+z^{\alpha_{^{2}}}+t^{\alpha_{3}}=0\},
\]
 
\[
X_{2}=\{x+y^{d-1}(x^{d}+z^{\alpha_{^{2}}})+t^{\alpha_{3}}=0\},
\]

considering the equivariant isomorphisms $\psi_{1}$ and $\psi_{2}$,
respectively, in the proof.
\end{prop}
\begin{proof}
Applying Theorem \ref{theoreme famiille unif}, $X_{1}$ corresponds
to the choice $d=1$ and $p(v)=v+v^{k}$, and $X_{2}$ corresponds
to the choice $d\geq2$ and $p(v)=v+v^{d}$.

1) An explicit isomorphism $\psi_{1}:X_{1}\setminus V(1+(xy)^{d-1})\rightarrow\mathbb{A}^{3}\setminus V(1+(YZ^{\alpha_{2}}+YT^{\alpha_{3}})^{d-1})$
is given by: 

\[
\psi_{1}:\left(\begin{array}{c}
x\\
y\\
z\\
t
\end{array}\right)\rightarrow\left(\begin{array}{c}
Y\\
Z\\
T
\end{array}\right)=\left(\begin{array}{c}
\frac{-y}{1+(xy)^{d-1}}\\
z\\
t
\end{array}\right).
\]

\begin{flushleft}
Its inverse $\psi_{1}^{-1}$ is given by: 
\par\end{flushleft}
\[
\psi_{1}^{-1}:\left(\begin{array}{c}
Y\\
Z\\
T
\end{array}\right)\rightarrow\left(\begin{array}{c}
x\\
y\\
z\\
t
\end{array}\right)=\left(\begin{array}{c}
-\frac{Z^{\alpha_{2}}+T^{\alpha_{3}}}{1+(YZ^{\alpha_{2}}+YT^{\alpha_{3}})^{d-1}}\\
-Y(1+(YZ^{\alpha_{2}}+YT^{\alpha_{3}})^{d-1})\\
Z\\
T
\end{array}\right)
\]

2) An explicit isomorphism $\psi_{2}:X_{2}\setminus V(1+(xy)^{d-1})\rightarrow\mathbb{A}^{3}\setminus V(1+(Y^{d}Z^{\alpha_{2}}+YT^{\alpha_{3}})^{d-1})$
is given by:

\[
\psi_{2}:\left(\begin{array}{c}
x\\
y\\
z\\
t
\end{array}\right)\rightarrow\left(\begin{array}{c}
Y\\
Z\\
T
\end{array}\right)=\left(\begin{array}{c}
\frac{-y}{1+(xy)^{d-1}}\\
z\\
t
\end{array}\right).
\]

\begin{flushleft}
Its inverse $\psi_{2}^{-1}$ is given by:
\par\end{flushleft}
\[
\psi_{2}^{-1}:\left(\begin{array}{c}
Y\\
Z\\
T
\end{array}\right)\rightarrow\left(\begin{array}{c}
x\\
y\\
z\\
t
\end{array}\right)=\left(\begin{array}{c}
-Y^{d-1}Z^{\alpha_{2}}-\frac{T^{\alpha_{3}}}{1+(Y^{d}Z^{\alpha_{2}}+YT^{\alpha_{3}})^{d-1}}\\
-Y(1+(YZ^{d\alpha_{2}}+YT^{\alpha_{3}})^{d-1})\\
Z\\
T
\end{array}\right)
\]
\end{proof}
\begin{thm}
\label{The-Koras-Russell-threefolds gm rational-1}All Koras-Russell
threefolds of the first kind $\{x+x^{k}y+z^{\alpha_{^{2}}}+t^{\alpha_{3}}=0\}$
in $\mathbb{A}^{4}=\mathrm{Spec}(\mathbb{C}[x,y,z,t])$ are $\mathbb{G}_{m}$-linearly
uniformly rational, considering the equivariant isomorphism $\psi$
in the proof.
\end{thm}
\begin{proof}
Let $X=\{x+x^{k}y+z^{\alpha_{^{2}}}+t^{\alpha_{3}}=0\}$ be a Koras-Russell
threefold of the first kind, let $\mathcal{U}$ be the principal open
subset of $X$ where $x$ does not vanish and let $\mathcal{V}$ be
the principal open subset of $X$ where $1+yx^{d-1}$ does not vanish.
The principal open subsets $\mathcal{U}=X_{x}$ and $\mathcal{V}=X_{1+yx^{d-1}}$
form a covering of $X$ by $\mathbb{G}_{m}$-stable open subsets. 

Since $\varGamma(\mathcal{U},\mathcal{O_{U}})=\mathbb{C}[x,x^{-1},y,z,t]/(x+x^{k}y+z^{\alpha_{^{2}}}+t^{\alpha_{3}})\simeq\mathbb{C}[x,x^{-1},z,t]$,
$X$ is $\mathbb{G}_{m}$-linearly rational at every point of $\mathcal{U}$.

By Proposition \ref{cover-1}, we have an explicit $\mathbb{G}_{m}$-equivariant
isomorphism between an open neighborhood of the fixed point in $X_{1}=\{x+x^{k}y^{k-1}+z^{\alpha_{^{2}}}+t^{\alpha_{3}}=0\}$
and an open subset of $\mathbb{A}^{3}$. Moreover $X_{1}$ admits
an action of the cyclic group $\mu_{k-1}$ given by $\epsilon\cdot(x,y,z,t)\rightarrow(x,\epsilon y,z,t)$
such that the action of $\mu_{k-1}$ factor through that of $\mathbb{G}_{m}$.
Thus the quotient for the action of the cyclic group and the isomorphism
obtained in Proposition \ref{cover-1} commute. In this case the quotient
of $\mathbb{A}^{3}$ for the action of $\mu_{k-1}$ is still isomorphic
to $\mathbb{A}^{3}$. Since $X_{1}/\!/\mu_{k-1}\simeq X$, the $\mathbb{G}_{m}$-equivariant
map $\psi_{1}$ given in Proposition \ref{cover-1} descends to a
$\mathbb{G}_{m}$-equivariant isomorphism $\psi$:

\[
\xymatrix{X_{1}\setminus V(1+(xy)^{k-1})\ar[rr]^{\psi}\ar[d]^{/\!/\mu_{k-1}} &  & \mathbb{A}^{3}\setminus V(1+(YZ^{\alpha_{2}}+YT^{\alpha_{3}})^{k-1})\ar[d]^{/\!/\mu_{d-1}}\\
X\setminus V(1+yx^{k-1})\ar[rr]^{\psi} &  & \mathbb{A}^{3}\setminus V(1+Y(Z^{\alpha_{2}}+T^{\alpha_{3}})).
}
.
\]
\\
\end{proof}
\begin{rem}
The variety $X$ is endowed with an hyperbolic $\mathbb{G}_{m}$-action,
the $\mathbb{G}_{m}$-stable principal open subset $\mathcal{V}=X_{1+yx^{d-1}}$
is isomorphic to a principal open subset of $\mathbb{A}^{3}$ endowed
with an hyperbolic $\mathbb{G}_{m}$-action but the $\mathbb{G}_{m}$-stable
principal open subset $\mathcal{U}=X_{x}$ is isomorphic to a principal
open subset of $\mathbb{A}^{3}$ endowed with a $\mathbb{G}_{m}$-action
with positive weights only.
\end{rem}
\begin{prop}
The Koras-Russell threefolds of the second kind given by the equations
\[
X=\{x+y(x^{d}+z^{\alpha_{^{2}}})^{l}+t^{\alpha_{3}}=0\},
\]
 in $\mathbb{A}^{4}=\mathrm{Spec}(\mathbb{C}[x,y,z,t])$ with $l=1$
or $l=2$ or $d=2$ are $\mathbb{G}_{m}$-linearly uniformly rational.
\end{prop}
\begin{proof}
In the case $l=1$ we consider the $\mathbb{G}_{m}$-uniformly rational
variety: 
\[
X_{2}=\{x+y^{d-1}(x^{d}+z^{\alpha_{^{2}}})+t^{\alpha_{3}}=0\},
\]
 given in Proposition \ref{cover-1}. The cyclic group $\mu_{d-1}$
on $X_{2}$ via $\epsilon\cdot(x,y,z,t)\rightarrow(x,\epsilon y,z,t)$,
and this action factors through that of $\mathbb{G}_{m}$. Thus the
quotient for the action of cyclic group and the isomorphism obtains
in Proposition \ref{cover-1} commute. The conclusion follows by the
same method as in the proof of Theorem \ref{The-Koras-Russell-threefolds gm rational-1}. 

\noindent

Let $X_{d-1}=\{x+y^{dl-1}(x^{d}+z^{\alpha_{^{2}}})^{l}+t^{\alpha_{3}}=0\}\rightarrow X$
be the cyclic cover of order $dl-1$ of $X$ branched along the divisor
$\{y=0\}$. The A-H presentation of $X_{d-1}$(see \cite{P1}) is
$\mathbb{S}(\tilde{\mathbb{A}}_{(u,v^{d})}^{2},\mathcal{D})$ with:
\[
\mathcal{D}=\left\{ \frac{a}{\alpha_{2}}\right\} D_{\alpha_{3}}+\left\{ \frac{b}{\alpha_{3}}\right\} D_{\alpha_{2}}+\left[0,\frac{1}{\alpha_{2}\alpha_{3}}\right]E,
\]
where $E$ is the exceptional divisor of the blow-up $\pi:\tilde{\mathbb{A}}_{(u,v^{d})}^{2}\rightarrow\mathbb{A}^{2}\simeq\mathrm{Spec}(\mathbb{C}[u,v])\simeq\mathrm{Spec}(\mathbb{C}[y^{d}z^{\alpha_{2}},yx])$
, and where $D_{\alpha_{2}}$ and $D_{\alpha_{3}}$ are the strict
transforms of the curves $L_{1}=\left\{ v+(u+v^{d})^{l})=0\right\} $
and $L_{2}=\left\{ u=0\right\} $ in $\mathbb{A}^{2}=\mathrm{Spec}(\mathbb{C}[u,v])$
respectively, $(a,b)\in\mathbb{Z}^{2}$, being chosen so that $ad\alpha_{3}+b\alpha_{2}=1$. 

First of all, variables $l$ and $d$ can be exchanged, just considering
the automorphism of $\mathbb{A}^{2}=\mathrm{Spec}(\mathbb{C}[u,v])$
which send $u$ on $u-(v-u^{l})^{d}$ and $v$ on $v-u^{l}$. Then
$v+(u+v^{d})^{l})$ is sent on $v$. From now it will be assume that
$l=2$.

By showing that $X_{d-1}$ is $\mathbb{G}_{m}$-linearly rational,
one can explicit a birational map between $X$ and $\mathbb{A}^{3}$.
This map will be an equivariant isomorphism between an open subset
of $X$ containing the fixed point and an open subset of $\mathbb{A}^{3}$.
The divisor $D=L_{1}+L_{2}$ is birationally equivalent to $D'=\{uv=0\}$
via be the birational endomorphism $\varphi$ of $\mathbb{A}^{2}=\mathrm{Spec}(\mathbb{C}[u,v])$
defined by $u\rightarrow\frac{u(1+(v-u^{2})^{2d-1})}{1-u(v-u^{2})^{d-1}}$
and $v\rightarrow v-u^{2}$. Thus $X_{d-1}$ is $\mathbb{G}_{m}$-linearly
rational. Moreover the application $\varphi$ is $\mu_{2d-1}$-equivariant,
considering the action of $\mu_{2d-1}$ given by $\epsilon\cdot(u,v)\rightarrow(\epsilon^{d}u,\epsilon v)$.
The desired result is now obtained by the same technique as in Theorem
\ref{The-Koras-Russell-threefolds gm rational-1}.
\end{proof}

\section{\label{sec:4}Examples of non $\mathbb{G}_{m}$-rational varieties }

\noindent

Since the property to be $G$-uniformly rational is more restrictive
than being only uniformly rational, it is not surprising that there
are smooth and rational $G$-varieties which are not $G$-uniformly
rational. In this section we will exhibit some $\mathbb{G}_{m}$-varieties
which are smooth and rational but not not $\mathbb{G}_{m}$-uniformly
rational. However it is not known if these varieties are uniformly
rational. This provides candidates to show that the uniform rationality
conjecture has a negative answer.
\begin{prop}
\label{prop: curve genus}Let $C\subset\mathbb{A}^{2}$ be a smooth
affine curve of positive genus passing through the origin with multiplicity
one and let $X$ be a $\mathbb{G}_{m}$-variety equivariantly isomorphic
to $\mathbb{S}(\tilde{\mathbb{A}}_{(u,v)}^{2},\mathcal{D})$ with
$\mathcal{D}=\left\{ \frac{1}{p}\right\} D+[0,\frac{1}{p}]E$, where
$E$ is the exceptional divisor of the blow-up and $D$ is the strict
transform of $C$. Then $X$ is a smooth rational $\mathbb{G}_{m}$-variety
but not a $\mathbb{G}_{m}$-uniformly rational variety.
\end{prop}
\begin{proof}
This is a direct consequence of the classification of hyperbolic $\mathbb{G}_{m}$-actions
on $\mathbb{A}^{3}$ given in Proposition \ref{classification A3}.
In this case the irreducible components of the support of the ps-divisors
are all rational. But the variety $\mathbb{S}(\tilde{\mathbb{A}}_{(u,v)}^{2},\mathcal{D})$
given in \cite[proposition 3.1]{P1} admits the support of $D$ in
the support of its ps-divisors. As the support of $D$ is not rational
it follows that the varieties obtained by this construction are not
$\mathbb{G}_{m}$-linearly rational and thus not $\mathbb{G}_{m}$-uniformly
rational since the two properties are equivalent in the case of $\mathbb{G}_{m}$-varieties
of complexity two (see Theorem \ref{GLUR=00003DGUR}).
\end{proof}
\begin{example}
Let $V(h)$ be a smooth affine curve of positive genus passing through
the origin with multiplicity one. Then the hypersurface $\{h(xy,zy)/y+t^{p}=0\}$,
is stable in $\mathbb{A}^{4}=\mathrm{Spec}(\mathbb{C}[x,y,z,t])$,
for the linear $\mathbb{G}_{m}$-action given by $\lambda\cdot(x,y,z,t)=(\lambda^{p}x,\lambda^{-p}y,\lambda^{p}z,\lambda t)$.
This variety is smooth using the Jacobian criterion and rational as
its algebraic quotient is rational but not $\mathbb{G}_{m}$-uniformly
rational. 
\end{example}

\subsection{Numerical obstruction for rectifiability of curves}

For a $\mathbb{G}_{m}$-variety $\mathbb{S}(Y,\mathcal{D})$, the
non-rationality of the irreducible components of the support of $\mathcal{D}$
(see Proposition \ref{prop: curve genus}) is not the only obstruction
to being $\mathbb{G}_{m}$-rational. There exist divisors $D=L_{1}+L_{2}$
where $L_{i}$ is isomorphic to $\mathbb{A}^{1}$ for $i=1,2$ and
such that $D$ is not birationally equivalent to $D'=\{uv=0\}$. Such
$D$ can be used to construct $\mathbb{G}_{m}$-variety $\mathbb{S}(Y,\mathcal{D})$
where the irreducible components of the support of the ps-divisors
are all rational and such that $\mathbb{S}(Y,\mathcal{D})$ is not
$\mathbb{G}_{m}$-rational. To prove the existence of such $D$, we
will use an invariant, the \emph{Kumar-Murthy dimension} (see \cite{Ku-Mu}).
Recall that a pair $(X,D)$ is said \emph{smooth }if $X$ is a smooth
projective surface and $D$ is a SNC divisor on $X$. For every divisor
$D$ on a smooth projective variety, we define the \emph{Iitaka dimension,
}$\kappa(X,D):=sup\,dim\phi_{\left|nD\right|}(X)$ in the case where
$\left|nD\right|\neq\emptyset$ for some $n$, and $\kappa(X,D):=-\infty$
otherwise\emph{, }where $\phi_{\left|nD\right|}:X\dashrightarrow\mathbb{P}^{N}$
is the rational map associated to the linear system $\left|nD\right|$
on $X$.

\noindent
\begin{lem}
Let $D_{0}=\sum_{i=1}^{k}D_{i}$ be a reduced divisor on a complete
surface $X_{0}$, with $D_{i}$ irreducible for each $i\geq0$. For
any birational morphism $\pi:X\rightarrow X_{0}$ such that the pair
$(X,D_{X})$ is smooth, with $D_{X}$ the strict transform of $D$,
the value $\kappa(X,2K_{X}+D_{X})$ does not depend on the choice
of $\pi$.
\end{lem}
\begin{proof}
By the Zariski strong factorization Theorem, it suffices to show that
this dimension is invariant under blow-ups. Let $(X,D_{X})$ be a
resolution of the pair $(X_{0},D_{0})$ such that $X$ is smooth and
$D_{X}$ is SNC. Let $\pi:\tilde{X}\rightarrow X$ be the blow-up
of a point $p$ in $X$. Since $D_{X}$ is SNC, there are three possible
cases: $p\notin D_{X}$, $p$ is contained in a unique irreducible
component of $D_{X}$, or $p$ is a point of intersection of two irreducible
components $D_{X}$. We have then for any integer $n$: $n(2K_{\tilde{X}}+D_{\tilde{X}})=\pi^{*}(n(2K_{X}+D_{X}))+n(2-m)E,\:m=2,1,0$
respectively. Therefore, 
\[
\Gamma(X,\mathcal{O}(n(2K_{\tilde{X}}+D_{\tilde{X}})))=\Gamma(X,\mathcal{O}(\pi^{*}(n(2K_{X}+D_{X})+(2-m)E)))=\Gamma(X,\mathcal{O}(\pi^{*}(n(2K_{X}+D_{X})))),
\]
 and so, by the projection formula (\cite[II.5]{Har}), $\Gamma(X,\mathcal{O}(\pi^{*}(n(2K_{X}+D_{X}))))\simeq\Gamma(X,\mathcal{O}(n(2K_{X}+D_{X})))$
for any integer $n$.
\end{proof}
\begin{defn}
The \emph{Kumar-Murthy dimension} $k_{M}(X_{0},D_{0})$ of ($X_{0},D_{0})$
is the Iitaka dimension $\kappa(X,2K_{X}+D_{X})$ where $\pi:X\rightarrow X_{0}$
is any birational morphism such that the pair $(X,D_{X})$ is smooth.
\end{defn}
\noindent
\begin{defn}
A pair $(Y,D)$ (as in definition \ref{pars equivalence}) is \emph{birationally
rectifiable }if it is birationally equivalent to the union of $k\leq N=\mathrm{dim}(Y)$
general hyperplanes in $\mathbb{P}^{N}$. Note in particular that
$Y$ is rational and that the irreducible components of $D$ are either
rational or uniruled. 
\end{defn}
Since, the Kumar-Murthy dimension of the pair $(\mathbb{P}^{2},D)$,
where $D$ is a union of two distinct lines, is equal to $-\infty$,
we obtain:
\begin{prop}
\label{condition dimension}If a reduced divisor $D=D_{1}+D_{2}$
in $\mathbb{P}^{2}$ is birationally rectifiable then $k_{M}(\mathbb{P}^{2},D)=-\infty$.
\end{prop}
\begin{example}
\label{non rectification-1} Let $C=\{u+(v+u^{2})^{2}=0\}$ and $C'=\{\alpha v(v-\beta)+u=0\}$
be two curves in $\mathbb{A}^{2}=\mathrm{Spec}(\mathbb{C}[u,v])$
where $(\alpha,\beta)\in\mathbb{C}^{2}$ are generic parameters chosen
such that $C$ and $C'$ intersect normally. Let $\bar{C}$ and $\bar{C}'$
be the closures in $\mathbb{P}^{2}$ of $C$ and $C'$ respectively
and let $D=\bar{C}+\bar{C}'$. Then:

i) $C$ and $C'$ are isomorphic to $\mathbb{A}^{1}$ 

ii) $k_{M}(\mathbb{P}^{2},D)\neq-\infty$. 
\end{example}
\begin{proof}
The curve $C'$ is clearly isomorphic to $\mathbb{A}^{1}$. In the
case of $C$, consider the following two automorphisms:

$\psi_{1}:(u,v)\rightarrow(u,v+u^{2})$ and $\psi_{2}:(u,v)\rightarrow(u+v^{2},v)$
then the composition \foreignlanguage{english}{$\psi_{2}\circ\psi_{1}:\begin{cases}
u & \rightarrow u+(u+v^{2})^{2}\\
v & \rightarrow v+u^{2}
\end{cases}$} sends $C$ on a coordinate axe. A minimal log-resolution $\pi:S_{7}\rightarrow\mathbb{P}^{2}$
of $\bar{C}\cup\bar{C'}$ is obtained by performing a sequence of
seven blow-ups, five of them with centers lying over the singular
point of $\bar{C}$ and the remaining two over the singular point
of $\bar{C'}$.

\selectlanguage{french}%
\begin{figure}[H]
\includegraphics[scale=0.45]{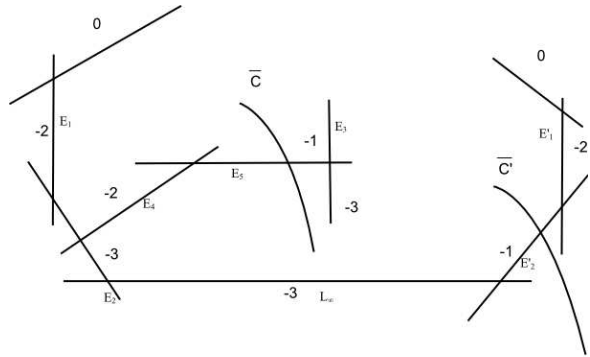}

\selectlanguage{american}%
\caption{Resolution of $(\mathbb{P}^{2},(\bar{C}+\bar{C'})$, the divisors
$E_{i}$ and $E_{i}'$ are exceptional divisors obtained blowing-up
$\bar{C}\cap L_{\infty}$ and $\bar{C}'\cap L_{\infty}$ numbered
according to the order of their extraction. }
\selectlanguage{french}%
\end{figure}

\selectlanguage{american}%
Using the ramification formula for the successive blow-ups occurring
in $\pi$, we find that the canonical divisor of $S_{7}$ is equal
to $K_{S_{7}}=-3l+E_{1}+2E_{2}+3E_{3}+6E_{4}+10E_{5}+E'_{1}+2E'_{2}$,
where $l$ denotes the proper transform of a general line in $\mathbb{P}^{2}$.
The total transform of the divisor $\bar{C}+\bar{C'}$ is equal to
$\pi^{*}(\bar{C}+\bar{C'})=\bar{C}+2E_{1}+4E_{2}+6E_{3}+11E_{4}+18E_{5}+\bar{C'}+E'_{1}+2E'_{2}$,
where we have identified $\bar{C}$ and $\bar{C}'$ with their proper
transforms in $S_{7}$.

Since $\bar{C}$ is of degree $4$ and $\bar{C'}$ is of degree $2$,
the proper transform of $\bar{C}+\bar{C}'$ in $S_{7}$ is linearly
equivalent to $6l$ and we obtain 
\[
2K_{S_{7}}+D=2K_{S_{7}}+\pi^{*}(\bar{C}+\bar{C'})-(2E_{1}+4E_{2}+6E_{3}+11E_{4}+18E_{5}+E'_{1}+2E'_{2})=E_{4}+2E_{5}+E'_{1}+2E'_{2},
\]

\begin{flushleft}
which is an effective divisor. Thus $k_{M}(\mathbb{P}^{2},D)\neq-\infty$,
and by Proposition \ref{condition dimension}, $D$ is not birationally
rectifiable.
\par\end{flushleft}
\end{proof}

\subsection{Application }

Let $X$ be the subvariety of $\mathbb{A}^{5}=\mathrm{Spec}(\mathbb{C}[w,x,y,z,t])$
defined by the two equations $\{w+y(x+yw^{2})^{2}+t^{\alpha_{3}}=0\}$
and $\{\alpha x(yx-\beta)+w+z^{\alpha_{2}})=0\}$, where $(\alpha,\beta)\in\mathbb{C}^{2}$
are the same parameters as in Example \ref{non rectification-1}.
This variety is endowed with a hyperbolic $\mathbb{G}_{m}$-action
induced by the linear one on $\mathbb{A}^{5}$, $\lambda\cdot(w,x,y,z,t)=(\lambda^{\alpha_{2}\alpha_{3}}w,\lambda^{\alpha_{2}\alpha_{3}}x,\lambda^{-\alpha_{2}\alpha_{3}}y,\lambda^{\alpha_{3}}z,\lambda^{\alpha_{2}}t)$.
Moreover is equivariantly isomorphic to the hypersurface in $\mathbb{A}^{4}=\mathrm{Spec}(\mathbb{C}[x,y,z,t])$
defined by $\{z^{\alpha_{2}}-\alpha x(xy-\beta)+y(x+y(z^{\alpha_{2}}-\alpha x(xy-\beta))^{2})^{2}+t^{\alpha_{3}}=0\}$.
\begin{thm}
The threefold $X$ is a smooth rational $\mathbb{G}_{m}$-variety
but not a $\mathbb{G}_{m}$-uniformly rational variety. 
\end{thm}
\begin{proof}
The A-H presentation of $X$ is given by $\mathbb{S}(\tilde{\mathbb{A}}_{(u,v)}^{2},\mathcal{D})$
with 
\[
\mathcal{D}=\left\{ \frac{a}{\alpha_{2}}\right\} D{}_{1}+\left\{ \frac{b}{\alpha_{3}}\right\} D{}_{2}+\left[0,\frac{1}{\alpha_{2}\alpha_{3}}\right]E,
\]
where $E$ is the exceptional divisor of the blow-up $\pi:\tilde{\mathbb{A}}_{(u,v)}^{2}\rightarrow\mathbb{A}^{2}$,
$D_{1}$ and $D_{2}$ are the strict transform of the curves $C$
and $C'$ of Example \ref{non rectification-1}, and $(a,b)\in\mathbb{Z}^{2}$
are such that $a\alpha_{3}+b\alpha_{2}=1$. The presentation comes
from the fact that $X$ is endowed with an action of $\mu_{\alpha_{2}}\times\mu_{\alpha_{3}}$
factoring through that of $\mathbb{G}_{m}$ and given by $(\epsilon,\xi)\cdot(x,y,z,t)\rightarrow(x,y,\epsilon z,\xi t)$.
\[
\xymatrix{ & X\ar[dr]\ar[dl]\\
X/\!/\mu_{\alpha_{2}} &  & X/\!/\mu_{\alpha_{3}}
}
\]

\begin{flushleft}
By \cite[Example 3.1]{P1}, $X/\!/\mu_{\alpha_{2}}$ is equivariantly
isomorphic to $\mathbb{S}(\tilde{\mathbb{A}}_{(u,v)}^{2},\left\{ \frac{1}{\alpha_{3}}\right\} D{}_{2}+\left[0,\frac{1}{\alpha_{3}}\right]E)$
and $X/\!/\mu_{\alpha_{3}}$ is equivariantly isomorphic to $\mathbb{S}(\tilde{\mathbb{A}}_{(u,v)}^{2},\left\{ \frac{1}{\alpha_{2}}\right\} D{}_{1}+\left[0,\frac{1}{\alpha_{2}}\right]E)$.
In fact, $X/\!/\mu_{\alpha_{2}}$ is equivariantly isomorphic to $\mathbb{A}^{3}=\mathrm{Spec}(\mathbb{C}[x,y,t])$
with the $\mathbb{G}_{m}$-action defined via $\lambda\cdot(x,y,t)=(\lambda^{\alpha_{3}}x,\lambda^{-\alpha_{3}}y,\lambda t)$
and $X/\!/\mu_{\alpha_{3}}$ is equivariantly isomorphic to $\mathbb{A}^{3}=\mathrm{Spec}(\mathbb{C}[x,y,z])$
with the $\mathbb{G}_{m}$-action defined via $\lambda\cdot(x,y,z)=(\lambda^{\alpha_{2}}x,\lambda^{-\alpha_{2}}y,\lambda z)$.
In particular $X$ is a Koras-Russell threefolds (see \cite{K-R,Ka-K-ML-R,P1}).
Now the result follows from Proposition \ref{condition dimension}
and Example \ref{non rectification-1}.
\par\end{flushleft}
\end{proof}

\section{\label{sec:5}Weak equivariant rationality}

\noindent

The property to be $G$-linearly uniformly rational is very restrictive.
We will now introduce a weaker notion:
\begin{defn}
A $G$-variety $X$ is called \emph{weakly $G$-rational} at a point\emph{
$x$} if there exist an open $G$-stable neighborhood $U_{x}$ of
$x$, an open subvariety $V$ of $\mathbb{A}^{n}$ equipped with a
$G$-action and a $G$-equivariant isomorphism between $U_{x}$ and
$V$. We said that $X$ is\emph{ weakly} \emph{$G$-uniformly rational}
if it is\emph{ }weakly $G$-rational\emph{ }at every point.
\end{defn}
Note that, in contrast with Definition \ref{definition equivariant}
ii) we only require that $V\subset\mathbb{A}^{n}$ is endowed with
a $G$-action, in particular it need not be the restriction of a $G$-action
on $\mathbb{A}^{n}$. In summary we have a sequence of implications
between these different notions of $G$-rationality: $G$-linearly
uniformly rational implies $G$-uniformly rational which implies $G$-weakly
uniformly rational, which finally implies uniformly rational.
\begin{thm}
Let $S\subset\mathbb{A}^{3}=\mathrm{Spec}(\mathbb{C}[x,y,z])$ be
the surface defined by the equation $z^{2}+y^{2}+x^{3}-1=0$, equipped
with the $\mu_{2}$-action $\tau\cdot(x,y,z)\rightarrow(x,y,-z)$
on $\mathbb{A}^{3}$, where $\tau$ is the non trivial element of
$\mu_{2}$. Then $S$ is weakly $\mu_{2}$-uniformly rational but
not $\mu_{2}$-uniformly rational.
\end{thm}
\begin{proof}
The surface $S$ is the cyclic cover of $\mathbb{A}^{2}$ of order
$2$ branched along the smooth affine elliptic curve $C=\{y^{2}+x^{3}-1=0\}\subset\mathbb{A}^{2}$.
By construction, the inverse image of $C$ in $S$ is equal to the
fixed points set of the involution. It follows that $S$ is not $\mu_{2}$
rational at the point $p=(1,0,0)$. Indeed, every $\mu_{2}$-action
on $\mathbb{A}^{2}$ being linearisable (see \cite[Theorem 4.3]{Kam}),
its fixed points set is rational. Therefore there is no $\mu_{2}$-stable
open neighborhood of $p$ which is equivariantly isomorphic to a stable
open subset of $\mathbb{A}^{2}$ endowed with a $\mu_{2}$-action.
However, there is an open subset $U$ of $\mathbb{A}^{2}$ which can
be endowed with a $\mu_{2}$-action such that $U$ is equivariantly
isomorphic to a $\mu_{2}$-stable open neighborhood of $p$.

Indeed, letting $u=z+y$ and $v=z-y$, $S$ is isomorphic to the surface
defined in $\mathbb{A}^{3}=\mathrm{Spec}(\mathbb{C}[u,v,x])$ by the
equation $\{uv-x^{3}+1=0\}$. The open subset $V_{1}=S\setminus\{1+x+x^{2}=u=0\}$
is isomorphic to $\mathbb{A}^{2}$ with coordinates $u$ and $v/(1+x+x^{2})=(x-1)/u=w$.
Let $V=S\setminus\{1+x+x^{2}=0\}$ be an open subset in $V_{1}$,
and let $x=uw+1$, then $V$ has the following coordinate ring:

\[
\mathbb{C}\left[u,w,\frac{1}{(uw+1)^{2}+uw+1+1}\right]=\mathbb{C}\left[u,w,\frac{1}{(uw)^{2}+3uw+3}\right].
\]

\begin{flushleft}
The open subset $V$ contains the point $p$ and is stable by the
action of $\mu_{2}$ defined by: 
\par\end{flushleft}
\[
\tau\cdot(u,v)=(w((uw)^{2}+3uw+3),u((uw)^{2}+3uw+3)^{-1}).
\]
 
\begin{flushleft}
So $S$ is $\mu_{2}$-weakly rational but not $\mu_{2}$-rational.
\par\end{flushleft}
\end{proof}

\end{document}